\documentclass[12pt]{article}

\usepackage{fullpage}
\usepackage{amsfonts}
\usepackage{enumitem}
\usepackage{amsmath}
\usepackage{multirow}
\usepackage{array}
\usepackage{amssymb}
\usepackage{amsthm}
\usepackage{hyperref}
\usepackage{mathrsfs}
\usepackage{cite}
\usepackage{aliascnt}
\usepackage{centernot}
\usepackage{dsfont}
\usepackage{stmaryrd}
\usepackage{tikz-cd}
\usepackage{longtable}
\usepackage{rotating}

\DeclareMathOperator{\GL}{GL}
\DeclareMathOperator{\PSp}{PSp}

\DeclareMathOperator{\SU}{SU}
\DeclareMathOperator{\GU}{GU}
\DeclareMathOperator{\SL}{SL}
\DeclareMathOperator{\PSL}{PSL}
\DeclareMathOperator{\Sp}{Sp}
\DeclareMathOperator{\PGL}{PGL}

\DeclareMathOperator{\PSO}{PSO}
\DeclareMathOperator{\PSU}{PSU}
\DeclareMathOperator{\Aut}{Aut}
\DeclareMathOperator{\Out}{Out}

\DeclareMathOperator{\lcm}{lcm}

\DeclareMathOperator{\Gal}{Gal}

\DeclareMathOperator{\diag}{diag}
\DeclareMathOperator{\adiag}{adiag}
\DeclareMathOperator{\soc}{soc}
\DeclareMathOperator{\Sz}{Sz}

\DeclareMathOperator{\PGammaL}{P\Gamma L}
\DeclareMathOperator{\PSigmaL}{P\Sigma L}
\DeclareMathOperator{\PGammaU}{P\Gamma U}
\DeclareMathOperator{\PSigmaU}{P\Sigma U}
\DeclareMathOperator{\POmega}{P\Omega}

\begin{document}
\setlength{\parindent}{15pt}

\newtheorem{thm}{Theorem}[section]
\newcommand{\thmautorefname}{Theorem}

\numberwithin{equation}{thm}

\newtheorem*{thm*}{Theorem}

\newaliascnt{prop}{thm}
\newtheorem{prop}[prop]{Proposition}
\aliascntresetthe{prop}
\newcommand{\propautorefname}{Proposition}

\newaliascnt{lem}{thm}
\newtheorem{lem}[lem]{Lemma}
\aliascntresetthe{lem}
\newcommand{\lemautorefname}{Lemma}

\newaliascnt{crl}{thm}
\newtheorem{crl}[crl]{Corollary}
\aliascntresetthe{crl}
\newcommand{\crlautorefname}{Corollary}

\newtheorem{thmintro}{Theorem}
\renewcommand\thethmintro{\Alph{thmintro}}
\newcommand{\thmintroautorefname}{Theorem}

\newaliascnt{qstintro}{thmintro}
\newtheorem{qstintro}[qstintro]{Question}
\renewcommand\theqstintro{\Alph{qstintro}}
\aliascntresetthe{qstintro}
\newcommand{\qstintroautorefname}{Question}

\newaliascnt{conj}{thm}
\newtheorem{conj}[conj]{Conjecture}
\aliascntresetthe{conj}
\newcommand{\conjautorefname}{Conjecture}

\theoremstyle{definition}
\newaliascnt{dfn}{thm}
\newtheorem{dfn}[dfn]{Definition}
\aliascntresetthe{dfn}
\newcommand{\dfnautorefname}{Definition}
\newaliascnt{rmk}{thm}
\newtheorem{rmk}[rmk]{Remark}
\aliascntresetthe{rmk}
\newcommand{\rmkautorefname}{Remark}

\newaliascnt{exm}{thm}
\newtheorem{exm}[exm]{Example}
\aliascntresetthe{exm}
\newcommand{\exmautorefname}{Example}

\newaliascnt{qst}{thm}
\newtheorem{qst}[qst]{Question}
\aliascntresetthe{qst}
\newcommand{\qstautorefname}{Question}

\newtheorem*{qst*}{Question}

\makeatletter
\def\Ddots{\mathinner{\mkern1mu\raise\p@
\vbox{\kern7\p@\hbox{.}}\mkern2mu
\raise4\p@\hbox{.}\mkern2mu\raise7\p@\hbox{.}\mkern1mu}}
\makeatother

\title{Finite groups with many elements of the same order}
\author{Ryan McCulloch, Lee Tae Young}

\maketitle

\begin{abstract}
We study a conjecture by Deaconescu on the solubility of finite groups with claims that if more than half of the elements in a finite group has the same order $k$, then the group is soluble. We show that the original conjecture fails by presenting some counterexamples. By restricting to a fixed $k$, the conjecture may or may not hold depending on $k$. We prove that if $k$ is a power of a prime other than $2$ or $3$, or if $k=2, 3$ or $4$, then the conjecture holds, while it fails for many other choices of $k$ including all multiples of $2$ and $3$ which are larger than $5$. For $k=4$ we also find the sharp upper bound of the ratio of elements of order $4$ in non-soluble groups. We also prove that for all $k>1$, it is always possible to find a finite non-soluble group where at least $2/15$ of the elements have order $k$.
\end{abstract}


\tableofcontents

\section{Introduction}

In this paper, we study the following conjecture: 
\begin{conj}[M. Deaconescu \upshape{\cite[21.43]{notebook}}]\label{conjecture}
Suppose that for a fixed positive integer $k$ at least half of the elements of a finite group $G$ have order $k$. Then $G$ is soluble.
\end{conj}

A motivation for this conjecture is the case $k=2$. It is a simple exercise that if at least $3/4$ of the elements of a finite group are involutions, then the group must be abelian. On the other hand, there are nonabelian groups with the ratio of involutions among its elements being arbitrarily close to $3/4$, namely the direct product of the dihedral group $D_8$ with a large elementary abelian $2$-group. Liebeck and MacHale \cite{LM} classified the finite groups in which at least half of the elements are involutions, and indeed, they are all solvable. Mann \cite{Mann} showed that for each $\epsilon\in (0,1)$, if at least $\epsilon$ of the elements of a finite group $G$ are involutions, then there exists a normal subgroup $N$ such that both the index $|G:N|$ and the order of the derived subgroup $|N'|$ are bounded by some function of $\epsilon$. Later, Potter \cite{Potter} and Berkovich \cite{Berkovich} independently extended the classification to the groups where more than $4/15$ of the elements are involutions.

The conjecture in its original form fails. We will present several counterexamples throughout the paper, including \autoref{direct_prod} and \autoref{construction2}. However, if we fix the number $k$, then the conjecture sometimes holds. For example, it is an elementary fact that the conjecture holds for $k=1$, and we saw above that it also holds for $k=2$. 

In this paper, we either prove the conjecture or give a counterexample for all prime power $k$. We also present counterexamples for many other values of $k$. We summarize our main results as the following theorem.
\begin{thm}\label{mainthm1}
{\upshape(1)} The conjecture holds when $k$ is a power of a prime $p>3$, and when $k=2, 3$ and $4$.\\{\upshape(2)} The conjecture fails when $k>4$ is divisible by either $2$ or $3$.\\{\upshape(3)} The conjecture also fails for some $k$ which are not divisible by $2$ and $3$.
\end{thm}  

The main difficulties of this problem for $k>2$, compared to the previously known case $k=2$, are the following. The works of Wall\cite{Wall} and Berkovich\cite{Berkovich} on the case $k=2$ use some properties of Frobenius-Schur indicator to translate the problem into a character theory problem. For $k\neq 2$, these properties fail for generalisations of Frobenius-Schur indicator, so their approach cannot be applied here. Liebeck and MacHale\cite{LM} used an alternative approach: they proved the existence of an inner automorphism which inverts more than half of the elements of a maximal abelian subgroup, and uses it to study the structure of the group. This approach also cannot be fully generalised to $k\neq 2$, since the proof depends on the fact that every subgroup is closed under taking inverses of elements. The analogous property for $k>2$ is that an automorphism $x$ does not necessarily invert an element $g$ in a subgroup but instead satisfies $gg^xg^{x^2}\cdots g^{x^{k-1}}=1$, but this does not imply that $g^x$ is also in the same subgroup. Instead of these approaches, we use an inequality (\autoref{self-normalising}) on the normalisers of certain subgroups such as the Sylow subgroups to prove the conjecture for certain $k$'s.

We also study the problem for ratios other than $1/2$. Specifically, we show that as long as $k\neq 1$, the conjecture fails if we replace the ratio $1/2$ by $2/15$. Also, for multiples of $2$ and $3$ except for their powers, the conjecture fails for ratios much larger than $1/2$.
\begin{thm}
Suppose that $k$ has a prime divisor $q\neq 2, 3$. Then \\{\upshape(1)} there exists a finite non-soluble group such that more than $(q-1)/(6q)$ of its elements have order $k$. Also, for $k=2$ and $3$, there exist finite non-soluble groups such that at least $4/15$ and $7/20$ of its elements have order $k$, respectively. \\{\upshape(2)} If in addition $k$ is divisible by $2$ or $3$, then there exists a finite non-soluble group such that more than $2(q-1)/(3q)$ or $3(q-1)/(4q)$, respectively, of its elements have order $k$.
\\{\upshape(3)} If $k$ is divisible by the exponent of a finite non-soluble group, then for any $0<\epsilon<1$, there exists a finite non-soluble group such that more than $\epsilon$ of its elements have order $k$.
\end{thm}
Note that in part (3), the condition that $k$ has a prime divisor other than $2$ and $3$ becomes redundant by Burnside's $pq$-theorem.

Finally, as a part of the proof of \autoref{mainthm1}, we found the sharp upper bound of the ratio of elements of order $4$ in a finite non-soluble group, which we state here as a separate theorem. 
\begin{thm}
The ratio of elements of order $4$ in a finite non-soluble group is bounded above by $7/15$, and this bound is sharp, achieved by $S_5\times C_4^n$ as $n\rightarrow \infty$.
\end{thm}

\section{Notations and basic properties}
In this section we fix some notations, and prove some basic properties including the reduction of the problem to finite non-soluble monolithic groups.

For a finite group $G$, let us denote by $G^n$ the direct product of $n$ copies of $G$. Also, for an element $g\in G$, let $o(g)$ denote the order of $g$. Let $C_d$ denote the cyclic group of order $d$, and let $A_n$ and $S_n$ denote the alternating and symmetric groups acting on $n$ points. The symbol $\mathbb{Z}_{>n}$ will denote the set of integers strictly larger than $n$. We will also use the following notation.

\begin{dfn}
For $k\in \mathbb{Z}_{>1}$ and a subset $S$ of a finite group $G$, let $\rho_k(S)$ and $\rho_k^*(S)$ be the ratio of elements of order dividing $k$ and exactly $k$ in $G$: $$\rho_k(S) := \frac{\left|\{g\in S : o(g)\text{ divides } k\}\right|}{|S|},\ \rho_k^*(S) := \frac{\left|\{g\in S : o(g)= k\}\right|}{|S|}.$$ Also, let $$\epsilon_k = \sup\{\rho_k(G)\mid G \text{ finite non-soluble}\}$$ and $$\epsilon_k^* = \max\{\rho_k(G)\mid G \text{ finite non-soluble}\}.$$ Note that if $\epsilon_k^*$ exists, then $\epsilon_k^* = \epsilon_k$. Also, for every group $G$ and every $k$, we have $\rho_k(G)>\rho_k^*(G)$ since the identity element is counted in $\rho_k(G)$ but not in $\rho_k^*(G)$.
\end{dfn}

We first present some counterexamples for \autoref{conjecture}. Our counterexamples will be based on the following elementary observation:

\begin{lem}\label{direct_prod}
Let $H$ be a finite group of exponent $e$. Then $\lim_{n\rightarrow \infty}\rho_e^*(H^n)=1$.
\end{lem}

\begin{proof}
Let $n\in \mathbb{Z}_{>0}$ and $G=H^n$. We may write $e = p_1^{a_1}p_2^{a_2}\cdots p_r^{a_r}$ for some pairwise distinct primes $p_1,\dots,p_r$ and positive integers $a_1,\dots,a_r$. Let $$T_j := \{h\in H \mid p_j^{a_j}\text{ divides } o(h)\}.$$ Note that $T_j$ is nonempty for each $j$ by the definition of exponent. 

Suppose that some element $g=(g_1,\dots,g_n)\in G$ has $o(g)\neq e$. Note that $o(g) = \lcm_{1\leq i\leq n}(o(g_i))$, so there exists some $j$ such that no $g_i$ is in $T_j$. Therefore, $$\left|\{g\in G\mid o(g)\neq e\}\right| \leq \sum_{j=1}^{r}(|G|-|T_j|)^n $$ so $$\frac{\left|\{g\in G\mid o(g)= e\}\right|}{|G|} = 1- \frac{\left|\{g\in G\mid o(g)\neq e\}\right|}{|G|} \geq 1 - \sum_{j=1}^{r}\left(1-\frac{|T_j|}{|H|}\right)^n.$$ The right-hand side converges to $1$ as $n\rightarrow \infty$.
\end{proof}

In particular, if $k$ is the exponent of a finite nonabelian simple group, then \autoref{direct_prod} provides a characteristically simple counterexample to \autoref{conjecture}. We also have the following construction which can be applied to more general choices of $k$:

\begin{prop}\label{construction2}
Let $\epsilon\in (0,1)$ and $k\in \mathbb{Z}_{>1}$. Suppose that there exists a finite non-soluble group $H$ such that $\rho_k(H)>\epsilon$. Then there exists a finite non-soluble group $G$ such that $\rho_k^*(G)>\epsilon$.
\end{prop}

\begin{proof}
By \autoref{direct_prod}, we can choose $G=H\times (C_k)^n$ with large $n$ so that $\rho_k^*((C_k)^n)>\epsilon/\rho_k(H)$. Note that if $(h,c)\in H\times (C_k)^n=G$ satisfies $o(h)\mid k$ and $o(c)=k$, then $o((h,c))=k$. Therefore $\rho_k^*(G)\geq \rho_k(H)\cdot \rho_k^*((C_k)^n)>\rho_k(H) \cdot (\epsilon/\rho_k(H)) = \epsilon$. 
\end{proof}

For example, $\rho_6(A_5)=36/60=3/5$, $\rho_{10}(A_5) = 40/60=2/3$, and $\rho_{15}(A_5)=45/60=3/4$, so for $k=6,10$ and $15$, we can find finite non-soluble groups whose $\rho^*_k$ are arbitrarily close to these numbers. In particular, \autoref{conjecture} fails for these $k$. Note that there is no nonabelian characteristically simple counterexample as in \autoref{direct_prod} for these $k$, since these are never the exponent of a non-soluble group by Burnside's $pq$-theorem.

Although \autoref{conjecture} fails in general, it might still hold if we fix $k$. As we saw earlier, it holds for $k=1$ and $2$. It might be interesting to reformulate \autoref{conjecture} as the following problem:
\begin{qst}\label{question}
(1) Can we compute $\epsilon_k$ for each $k\in \mathbb{Z}_{>0}$?\\
(2) For which $k$ is $\epsilon_k> 1/2$? \\
(3) For which $k$ does $\epsilon_k^*$ exist?
\end{qst}

The above examples can be summarized as the following:
\begin{prop}\label{examples_summary}
{\upshape(1)} $k$ is the exponent of some finite non-soluble group if and only if $\epsilon^*_k=1$.\\{\upshape(2)} For all $k$, $\sup\{\rho_k^*(G)\mid G \text{ finite non-soluble}\} = \epsilon_k$.\\{\upshape(3)} If $k$ divides $l$, then $\epsilon_k\leq \epsilon_l$.\\{\upshape(4)} $\epsilon_1=\epsilon_1^*=1/60$.\\{\upshape(5)} $\epsilon_2=\epsilon_2^*=4/15.$
\end{prop}
Note that \autoref{conjecture} for a fixed $k$ is equivalent to \autoref{question}(2) by \autoref{examples_summary}(2) and the fact $\rho_k^*<\rho_k$.
\begin{proof}
(1) For a finite group $G$, $\rho_k(G)=1$ if and only if $k$ is divisible by the exponent of $G$.

(2) Since $\rho_k^*(G)<\rho_k(G)$ for all $k>1$, we must have $\sup\{\rho_k^*(G)\mid G \text{ finite non-soluble}\} \leq \epsilon_k$. On the other hand, by \autoref{construction2}, for each finite non-soluble $G$, there exists a finite non-soluble $H$ such that $\rho_k^*(H)$ is arbitrarily close to $\rho_k(G)$. Therefore $\sup\{\rho_k^*(G)\mid G \text{ finite non-soluble}\} \geq \epsilon_k$.

(3) is immediate from the definition.
 
(4) follows immediately from the fact that the alternating group $A_5$ is a non-soluble group of the smallest order.

(5) follows from a result of Berkovich \cite[Theorem 3]{Berkovich} which proves that the only finite non-soluble groups $G$ with $\rho_2^*(G)\geq 1/4$ are $A_5\times E$ where $E$ is either trivial or an elementary abelian $2$-group.
\end{proof}

We list more observations, including a reduction to finite non-soluble monolithic groups with soluble quotients, i.e. finite groups $G$ with a unique minimal normal subgroup $N$, with an additional condition that $G/N$ is soluble.
\begin{lem}\label{quo}
Let $k\in \mathbb{Z}_{>0}$ and $N\unlhd G$. Then $\rho_k(G)\leq \rho_k(G/N)$. 
\end{lem}

\begin{proof}
For $g\in G$, if $(gN)^k = g^kN \neq N$, then $g^k\neq 1$, so the number of elements of order dividing $k$ in $G$ is at most the number of elements of the union of cosets $\bigcup(gN\mid (gN)^k=1)$. Since each coset has size $|N|$ and there are $\rho_k(G/N)|G/N|$ such cosets, we get $\rho_k(G)|G|\leq |N||\rho_k(G/N)||G/N|$.
\end{proof}

\begin{crl}\label{monolithic_reduction}
Let $k\in \mathbb{Z}_{>0}$ and $\epsilon\in (0,1)$. Suppose that there is a finite non-soluble group $G$ with $\rho_k(G)>\epsilon$. Then there exist (possibly trivial) normal subgroups $N\unlhd S\unlhd G$ such that $\rho_k(G/N)\geq \rho_k(G)>\epsilon$ and $S/N$ is the unique minimal normal subgroup of $G/N$, and $G/S$ is soluble with $\rho_k(G/S)>\epsilon$. In particular, to find $\epsilon_k$ we only need to check $\rho_k$ of finite non-soluble monolithic groups whose quotient by the socle is soluble. 
\end{crl}

\begin{proof}
The existence of a non-soluble monolithic quotient of a finite non-soluble group $G$, whose quotient by the socle is soluble, is an elementary fact, and we can apply \autoref{quo} to this quotient, and also to its quotient by the socle.
\end{proof}

The following example and \autoref{M10} suggest that reduction to simple groups or even to almost simple groups seems to be impossible, at least in an intuitive way, so we need to look at not only simple groups but also groups which are monolithic but not (almost) simple.

\begin{exm}
There exists an index $2$ subgroup $G$ of $A_5\wr C_2$ which has $\rho_{24}(G)=21/25$. This is strictly larger than $\rho_{24}(A_5)=3/5$ and $\rho_{24}(S_5)=4/5$, which are the only almost simple groups with socle $A_5$. Therefore there exists a monolithic group whose socle is $S^n$ for some nonabelian simple $S$ and $n\geq 1$, and whose $\rho_k$ for some $k$ is larger than $\rho_k$ of any almost simple group whose socle is $S$.
\end{exm}

\section{Lower bounds}
As we saw earlier, the conjecture is not true in general for any ratio. In this section, we provide some examples that establish some lower bounds for infinite families of numbers $k$, including some very general ones. 

First we get a very weak, obvious general lower bound for $\epsilon_k$.
\begin{prop}\label{Alt}
For all $k\in \mathbb{Z}_{\geq 5}$, $\epsilon_k>  \rho^*_k(S_k) \geq  1/k$. If $k$ is odd, then $\epsilon_k\geq  \rho_k(A_k) \geq 2/k+2/k!$. Consequently, $\epsilon_k\geq 2/p+2/p!$ for any prime $p\geq 5$ dividing $k$. If $k$ is an odd prime, then this is the largest $\rho_k$ we can get from symmetric and alternating groups.
\end{prop}
\begin{proof}
There are $(k-1)!$ $k$-cycles in $S_k$. They are contained in $A_k$ if and only if $k$ is odd. The second part follows from the first part and \autoref{examples_summary}(3). 

For the last part, let $p$ be a prime and $n\geq 5$. Then $\rho_p^*(S_n)$ is the number of elements of cycle types $(p), (p^2)=(p,p), \dots, (p^a)=(p,p,\dots,p)$ where $a$ is the largest number such that $ap\leq n$. Therefore we get $$\rho_p^*(S_n)=\sum_{m=1}^{a} \frac{1}{p^m(m!)(n-mp)!}.$$ Also, $$\rho_p^*(A_n)=2\sum_{m=1}^{a} \frac{1}{p^m(m!)(n-mp)!}\text{ if }p\neq 2\text{, and }\rho_2^*(A_n)=2\sum_{\substack{2\leq m\leq a \\m\text{ even}}} \frac{1}{2^m(m!)(n-2m)!}.$$  From this formula we can see that if $p\nmid n$, then $\rho_p^*(S_{n-1})\geq \rho_p^*(S_n)$, and if $p\mid n$, then $\rho_p^*(S_n)> \rho_p^*(S_{n+p})$. Therefore, the maximum value of $\rho_p$ we can get among the symmetric and alternating groups is $\rho_p(A_p) =\rho_p^*(A_p)+\rho_1^*(A_p)=2/p+2/p!$ for any prime $p\geq 5$.
\end{proof}

\begin{prop}\label{Alt_two_primes} Let $p> q$ be primes. The $n$ maximizing $\rho_{pq}^*(S_n)$ is $pq$, and $$\rho_{pq}^*(A_n)/2\leq \rho_{pq}^*(S_{pq}) =  \frac{1}{pq}+\sum_{x=1}^{q-1}\sum_{\substack{1\leq y\leq p(q-x)/q\\y\in \mathbb{Z}}} \frac{1}{p^x(x!)q^y(y!)(pq-xp-y q)!} < \frac{1}{pq} + \frac{q-1}{p^qq!}.   $$
\end{prop}
\begin{proof}
Let $a$ be the largest integer such that $ap\leq n$, and for $0\leq x\leq a$, let $b_x$ be the largest integer such that $b_xp\leq n-xp$. Also let $c_{x,y}$ be the largest integer such that $c_{x,y}pq\leq n-xp-y q$. As in the proof of \autoref{Alt}, we get $$\rho_{pq}^*(S_n) = \sum_{x=0}^{a}\sum_{y=0}^{b_x} \sum_{z=\max(1-xy, 0)}^{c_{x,y}}\frac{1}{p^x(x!)q^y(y!)(pq)^z(z!)(n-xp-y q-zpq)!}.$$
Again, the maximum value is achieved at $n=pq$.
\end{proof}

Usually, we have a much better lower bound.

\begin{prop}\label{PSL}
Let $q$ be a power of an odd prime $p$. \\
{\upshape(1)} The exponent of $\PSL_2(q)$ is $p(q+1)(q-1)/4$.\\
{\upshape(2)}
\begin{align*}
\rho_{(q-1)/2}(\PSL_2(q)) = \frac{q(q+1)(q-3)/2 + 2}{(q-1)q(q+1)},\ \rho_{(q+1)/2}(\PSL_2(q)) = \frac{q(q-1)^2/2 +2}{(q-1)q(q+1)}.
\end{align*}
In particular, for any small number $\delta>0$, there exists $M\in \mathbb{Z}_{>0}$ such that for all prime powers $q>M$, $\epsilon_{(q-1)/2}>1/2-\delta$ and $\epsilon_{(q+1)/2}>1/2-\delta$.\\
{\upshape(3)} If $q\equiv 1$ mod $4$, then \begin{align*}
\rho_{q+1}(\PSL_2(q)) = \frac{q^3+3q+4}{2(q-1)q(q+1)} > 1/2.
\end{align*} In particular, $\epsilon_{q+1}>1/2$, so \autoref{conjecture} fails when $k-1$ is a prime power which is $\equiv 1$ mod $4$. \\{\upshape(4)} If $q\equiv 3$ mod $4$, then \begin{align*}
\rho_{q-1}(\PSL_2(q)) = \frac{q^3-5q+4}{2(q-1)q(q+1)} < 1/2.
\end{align*}\\{\upshape(5)} 
\begin{align*}
&\rho_{p(q-1)/2}(\PSL_2(q)) = \frac{ q(q^2  +2q-3) }{2(q-1)q(q+1)} > 1/2,\ \rho_{p(q+1)/2}(\PSL_2(q)) = \frac{q(q+1)^2}{2(q-1)q(q+1)}>1/2,\\&\rho_{(q-1)(q+1)/4}(\PSL_2(q)) = \frac{q-2}{q}>1/2 \text{ for }q\geq 5.
\end{align*}
In particular, $\epsilon_{p(q-1)/2}, \epsilon_{p(q+1)/2}$ and $\epsilon_{(q-1)(q+1)/4}$ are larger than $1/2$, so \autoref{conjecture} fails when $k$ is of these forms.
\end{prop}

\begin{proof}
The order of $\PSL_2(q)$ is $|\PSL_2(q)| = (q-1)q(q+1)/2$. From the well-known character table of $\PSL_2(q)$, we know the number of elements of each order: \\$\bullet$ If $q\equiv 1$ mod $4$, then the group is partitioned into the following subsets: \{$1$ element of order $1$\}, \{$q^2-1$ elements of order $p$\}, \{$(q-1)^2q/4$ elements of order dividing $(q+1)/2$\}, \{$(q-5)q(q+1)/4$ elements of order dividing $(q-1)/2$\}, and \{$q(q+1)/2$ elements of order $2$(which are not included in the previous two types of elements)\}.\\$\bullet$ If $q\equiv 3$ mod $4$, then the group is partitioned into the following subsets: \{$1$ element of order $1$\}, \{$q^2-1$ elements of order $p$\}, \{$(q-3)q(q+1)/4$ elements of order dividing $(q-1)/2$\}, \{$(q-3)q(q-1)/4$ elements of order dividing $(q+1)/2$\}, and \{$q(q-1)/2$ elements of order $2$ (which are, again, not included in the previous two types of elements)\}. \\From these, we can compute the exponent and $\rho_{k}(\PSL_2(q))$ for various $k$'s.

Both numbers in (2) converge to $1/2$ from below as $q$ increases. Therefore, for any $\delta>0$, we can find a large enough $M$ such that whenever $q>M$, $\epsilon_k\geq \rho_k(\PSL_2(q)) > 1/2-\delta$ for both $k=(q-1)/2$ and $k=(q+1)/2$.
\end{proof}

\begin{rmk}\autoref{PSL} also provides some prime numbers $k$ such that $\epsilon_k>1/2 - \delta$ for very small $\delta>0$, namely the prime numbers of the form $(3^b-1)/2$, cf. \cite[A028491]{OEIS}. For example, $797161=(3^{13}-1)/2$ is a prime, and for this number, $$\epsilon_{797161}\geq \rho_{797161}(\PSL_2(3^{13}))=1270932117001/2541867422652=0.4999993727741\dots .$$

For general primes $p\geq 5$ not of this form, we get $\epsilon_p\geq \rho_p(\PSL_2(p)) = \frac{p^2}{(p-1)p(p+1)/2} = 2p/(p^2-1)$, which is very slightly better than the lower bound $2/p$ we got from \autoref{Alt}. Later in \autoref{PSigmaL2}, we will see a much better general lower bound, namely $(p-1)/(6p)$.
\end{rmk}

\begin{prop}\label{PSL_even}
Let $q = 2^a$. \\{\upshape(1)} The exponent of $\PSL_2(q)=\SL_2(q)$ is $2(q-1)(q+1)$. \\{\upshape(2)} \begin{align*}
&\rho_{q-1}(\PSL_2(q)) = \frac{(q-2)q(q+1) +2}{2(q-1)q(q+1)}<\frac{1}{2} ,\ \rho_{q+1}(\PSL_2(q)) = \frac{(q-1)q^2+2}{2(q-1)q(q+1)}<\frac{1}{2},\\&\rho_{2(q-1)}(\PSL_2(q)) = \frac{(q-1)q(q+2)}{2(q-1)q(q+1)}>\frac{1}{2},\ \rho_{2(q+1)}(\PSL_2(q)) = \frac{q^2(q+1)}{2(q-1)q(q+1)}>\frac{1}{2}, \\&\rho_{(q-1)(q+1)}(\PSL_2(q)) =\frac{q-1}{q} > 1/2\text{ when }a\geq 2.
\end{align*}
In particular, $\epsilon_k>1/2$ for $k=2(q-1), 2(q+1)$ and $(q-1)(q+1)$.
\end{prop}

\begin{proof}
$\PSL_2(q)$ can be partitioned into the following subsets: \{$1$ element of order $1$\}, \{$q(q+1)(q-2)/2$ elements of order dividing $q-1$\}, \{$q^2(q-1)/2$ elements of order dividing $q+1$\}, and \{$q^2-1$ elements of order $2$\}.
\end{proof}

Again, \autoref{PSL_even} shows that if $k$ is a prime of the form $2^a\pm 1$, then $\epsilon_k>1/2-\delta$ for some small $\delta>0$. Note that both \autoref{PSL} and \autoref{PSL_even} are not capable of proving that $\epsilon_k>1/2$ when $k$ is a prime power: \autoref{PSL}(5) and \autoref{PSL_even} require $k$ to have two divisors whose gcd is either $1$ or $2$, and \autoref{PSL}(3) requires $k$ to be $2$ mod $4$.

\begin{prop}\label{Suzuki}
Let $q = 2^{2a+1}$ and $r=2^a$ for some $a\in \mathbb{Z}_{>0}$. Then the exponent of the Suzuki group $\vphantom{a}^2B_2(q)=\Sz(q)$ is $4(q^2+1)(q-1)$, and
\begin{align*}
&\rho_{4(q^2+1)}(\Sz(q)) = \frac{q}{2(q-1)}>1/2,\ \rho_{4(q-1)}(\Sz(q)) = \frac{q^3+q-2}{2(q^3-q^2+q-1)}>1/2, \\&\rho_{(q+2r+1)(q-1)}(\Sz(q)) = \frac{  3q^3 - 6q^2 + 2qr + 3q -2r-4}{4(q-1)(q^2+1)}>1/2,\\&\rho_{(q-2r+1)(q-1)}(\Sz(q)) = \frac{  3q^3 - 6q^2 - 2qr + 3q +2r-4}{4(q-1)(q^2+1)}>1/2,\\&\rho_{q-1}(\Sz(q)) = \frac{q-2}{2(q-1)}+\frac{1}{q^2(q-1)(q^2+1)} =\frac{1}{2}-\frac{1}{2(q-1)}+\frac{1}{q^2(q-1)(q^2+1)},\\& \rho_{q^2+1}(\Sz(q)) = \frac{q^2-q}{2(q^2+1)} + \frac{1}{q^2(q-1)(q^2+1)} =  \frac{1}{2} - \frac{q+1}{2(q^2+1)} +\frac{1}{q^2(q-1)(q^2+1)} .
\end{align*}
Also, if $d$ is a divisor of $q\pm 2r +1 $ such that $$\frac{d-1}{4(q\pm 2r +1)}>\frac{1}{2(q-1)}-\frac{1}{q^2(q-1)(q^2+1)},$$ then $\rho_{d(q-1)}(\Sz(q))>1/2$. In particular, any divisor $d>3$ can be used, and $d=3$ can be used if it divides $q-2r+1$. Similarly $\rho_{d(q^2+1)}(\Sz(q))>1/2$ for any divisor $d>1$ of $q-1$.
\end{prop}

\begin{proof}
Recall the following information about $\Sz(q)$, reported by Suzuki himself \cite{Suzuki} when he discovered the group: \\$\bullet$ $|\Sz(q)| = q^2(q-1)(q^2+1)$.\\$\bullet$ It has $(q-2r+1)(q-1)q^2/4$ cyclic subgroups of order $q+2r+1$, and every two of them intersect trivially. Similarly, it has $(q+2r+1)(q-1)q^2/4$ cyclic subgroups of order $q-2r+1$, $q^2(q^2+1)/2$ cyclic subgroups of order $q-1$, and $q^2+1$ Sylow $2$-subgroups, and every two of these also trivially intersect.  (Note that $(q+2r+1)(q-2r+1)=q^2+1$.)\\$\bullet$ Consequently, it has $1$ element of order $1$, $(q+2r)(q-2r+1)(q-1)q^2/4$ elements of order dividing $q+2r+1$, $(q-2r)(q+2r+1)(q-1)q^2/4$ elements of order dividing $q-2r+1$, $(q-2)q^2(q^2+1)/2$ elements of order dividing $q-1$, $q(q^2+1)$ elements of orders $2$, and $(q^2-q-1)(q^2+1)$ elements of order $4$. \\$\bullet$ The number of nonidentity elements of order dividing $d$ for a divisor $d$ of $q+2r+1$ equals $(d-1)(q-2r+1)(q-1)q^2/4$, and a similar formula works for the divisors of $q-2r+1$ and $q-1$. 

Also note that $\gcd(q-1, q+2r+1) =  \gcd(q-1, 2r+2) = \gcd(2r^2-1,2r+2) = \gcd(2r+1,2r+2)=1$. Similarly $\gcd(q-1,q-2r+1) = \gcd(2r^2-1, 2r-2) = \gcd(2r-1,2r-2)=1$ and $\gcd(q+2r+1,q-2r+1) = \gcd(q+2r+1, 2^{a+2}) = 1$. From these information we can compute the $\rho_k$ for various $k$'s listed.
\end{proof}

The next group gives our first example of a prime power $k$ for which the conjecture fails.
\begin{exm}\label{M10}
The Mathieu group of degree 10, $\text{M}_{10}$, has: $$\epsilon_8 \geq \rho_8(\text{M}_{10})=31/45 > \rho_8(\Aut(A_6)) = 26/45>1/2>\rho_8(A_6)=17/45.$$ 
This example also tells us that we can get almost simple groups with better ratio than their socle.
\end{exm}

Our next group is $\PSigmaL_2$ over certain fields, which covers the cases $k=2q$ and $3q$ for all primes $q>3$, and even gives a general lower bound (which is much smaller than $1/2$ but larger than $1/7$) for all numbers having a prime divisor $\geq 5$. Although the case $k=9$, which would result from $p=q=3$ in the statement of \autoref{PSigmaL2} below, falls into the exceptional cases we left out for some technical difficulties, one can easily verify that $\epsilon_9\geq \rho_9(\PSigmaL_2(3^3))=191/364>1/2$. We believe that the proof can be improved to include some of such exceptions. 

\begin{lem}\label{trivial-intersection}
Let $G=N\rtimes K$, and let $H$ be a subgroup of $N$ normalised by $K$ such that $\left(H\rtimes K\right)\cap \left(H\rtimes K\right)^g \leq N$ for any $g\notin H\rtimes K$. Then for any $k\in \mathbb{Z}_{>0}$, we have $$\rho_{k}(G) = \rho_{k}(H\rtimes K) - \frac{1}{|K|} \rho_{k}(H) + \frac{1}{|K|} \rho_{k}(N).$$
\end{lem}

\begin{proof}
The conditions imply that every two different $G$-conjugates of $H\rtimes K$ can only intersect in $N$. Therefore, the conjugates of the subsets $H\rtimes K - H = H\rtimes (K-\{1\})$ are pairwise disjoint. Also, if $(H\rtimes K - H)^{g_1} = (H\rtimes K - H)^{g_2}$ for some $g_1, g_2\in G$, then $(H\rtimes K)^{g_1} = (H\rtimes K)^{g_2}$. Since there are $|G:H\rtimes K|$ distinct conjugates of $H\rtimes K$, the union of the conjugates of $H\rtimes K - H$ has $|H\rtimes K-H|\cdot |G:H\rtimes K| = \left(|K|-1\right)|G|/|K| = |G|-|N|$ elements. Therefore, $N$ together with the conjugates of $H\rtimes K-H$ form a partition of $G$, so the number of the elements of order dividing $k$ in $G$ can be counted separately for each part: $$|G|\rho_k(G) =  \frac{|G|}{|H\rtimes K|} \left(|H\rtimes K|\rho_k(H\rtimes K) - |H|\rho_k(H)\right)  +  |N|\rho_k(N).$$
\end{proof}

\begin{lem}\label{commuting}
Let $G=N\rtimes K$, and let $H\leq \mathbf{C}_N(K)$. Suppose that $|H|$ and $|K|$ are relatively prime. Then for any $g\in G$ such that $\mathbf{C}_K(g)=1$, $\left(H\times K\right)\cap\left(H\times K\right)^g\leq N$. In particular, if $H\leq \mathbf{C}_N(k)\leq \mathbf{N}_N(H)$ for every nonidentity $k\in K$, then every $g\notin \mathbf{N}_G(H\times K)$ satisfies this condition.
\end{lem}

\begin{proof}
Suppose that $g\in G$, $h\in H$, and $a\in K$ be elements such that $(h,a)^g\in H\times K$. We may write $g=(n,b)$ for some $b\in K$ and $n\in N$. Then (with the understanding that $K$ acts on $N$ on the left)$$(h,a)^g = (n,b)^{-1}(h,a)(n,b) = \left(\vphantom{a}^{b^{-1}}\left(n^{-1}\right), b^{-1}\right)(h\ ^{a}n, ab) = \left(\vphantom{a}^{b^{-1}}\left(n^{-1}h\ ^an\right), b^{-1}ab\right)$$ so $n^{-1}h\ ^an\in H$. Since $h\in H\leq \mathbf{C}_N(a)$, we get $\vphantom{a}^a\left(n^{-1}h\ ^an\right) = \vphantom{a}^a\left(n^{-1}\right)\ ^ah\ ^{a^2}n = \vphantom{a}^a\left(n^{-1}\right)h\ ^{a^2}n \in H$. It follows that for any $m\in \mathbb{Z}_{>0}$, $$\left(n^{-1}h\ ^an\right)^m = \left(n^{-1}h\ ^an\right)\ ^a\left(n^{-1}h\ ^an\right)\cdots \ ^{a^{m-1}}\left(n^{-1}h\ ^an\right) = n^{-1}h^m\left(\vphantom{a}^{a^m}n\right). $$ In particular, if we choose $m=\lcm\left(o(h),o(n^{-1}h\ ^an)\right)$, then we get $1 = n^{-1}\ ^{a^m}n$, so that $n=\vphantom{a}^{a^m}n$. Therefore $a^m\in \mathbf{C}_K(n)=1$. Since $m$ divides the exponent of $|H|$, it is relatively prime to $|K|$, we get $a=1$ and $(h,a)=(h,1)\in H\leq N$.

If $H\leq \mathbf{C}_N(k)\leq \mathbf{N}_N(H)$ for every nonidentity $k\in K$ and if $g\in G - \mathbf{N}_G(H\times K)$, then $g=(n,k)$ for some $n\in N-\mathbf{N}_N(H)$, so $\mathbf{C}_K(n)=1$.
\end{proof}

\begin{lem}\label{Classical_Groups}
Let $\Gamma$ be an almost simple (projective) classical group of Lie type. Let $G$ be the subgroup of inner and diagonal automorphisms of $\soc(\Gamma)$ in $\Gamma$, and $\Phi \cong \Gamma/G$ be the subgroup of (standard) field and graph automorphisms of $\soc(\Gamma)$ in $\Gamma$, so that $\Gamma = G\rtimes \Phi$. Let $H$ be a self-normalising subgroup of $G$ which is normalised by $\Phi$. Then \\{\upshape(1)} $H\rtimes \Phi$ is self-normalising in $\Gamma$.\\{\upshape(2)} If $H\leq \mathbf{C}_N(\Phi)$ and if $|\Phi|$ is relatively prime to $\left| H\right|$, then for any $g\in \Gamma$ such that $\mathbf{C}_{\Phi}(g)=1$, $(H\rtimes \Phi)\cap (H\rtimes \Phi)^g \subseteq H$.\\{\upshape(3)} In the situation of (2), if $\mathbf{C}_N(F)=H$ for all nonidentity $F\in \Phi$, then the subset $G$ together with the conjugates of $\left(H\rtimes \Phi\right) - H = H\times (\Phi-\{1\})$ form a partition of $\Gamma$. In particular, $$\rho^*_{kd}(\Gamma) = \frac{|\Phi|-1}{|\Phi|} \rho^*_k\left(H\right) +\frac{1}{|\Phi|} \rho^*_{kd}\left( G\right)$$ for each $d\in \mathbb{Z}_{>1}$ dividing $|\Phi|$.
\end{lem}

\begin{proof}
(1) If $\gamma = (g, F)\in \mathbf{N}_{\Gamma}(H\rtimes \Phi)$, then since $(1_G,F)$ also normalises $H\rtimes \Phi$, $(g,1)$ should normalise $H\rtimes \Phi$. Since $H^{(g,1)} \leq H\rtimes \Phi \cap G = H$, we get $g\in \mathbf{N}_{G}(H)=H$ and $\gamma \in H\rtimes \Phi$.

(2) This is immediate from \autoref{commuting}.

(3) By \autoref{commuting}, every $g\in G -  \left(H\times \Phi\right) = G-\mathbf{N}_\Gamma(H\times \Phi)$ satisfies the condition of part (2). Therefore, the conjugates of $(H\rtimes \Phi)-H$ are pairwise disjoint. There are $|\Gamma : \left(H\rtimes \Phi\right)|=|\Gamma : \mathbf{N}_\Gamma\left(H\rtimes \Phi\right)|$ distinct conjugates of this subset. Therefore, the union of the conjugates of this subset has $|\Gamma:\left(H\rtimes \Phi\right)| |H|\left(|\Phi|-1\right) = |\Gamma|-|G|$ elements. Since they are subsets of $\Gamma -G$,  it follows that these subsets and $G$ form a partition of $\Gamma$.
\end{proof}

\begin{crl}\label{PSigmaL2}
Let $p$, $q$ be primes such that $q$ does not divide $(p-1)p(p+1)$. Then for any divisor $d$ of $p-1$, $p$ or $p+1$, we have $$\rho_{dq}^*(\PGammaL_2(p^q)) = \frac{1}{q} \rho_{dq}^*(\PGL_2(p^q))+ \frac{q-1}{q}\rho_d^*(\PGL_2(p)).$$ The same formula holds if we replace $\PGammaL$ and $\PGL$ with $\PSigmaL$ and $\PSL$. In particular, if $q> 3$ then 
\begin{align*}
&\epsilon_{2q}\geq \rho_{2q}(\PGammaL_2(2^q)) = \frac{2^{q}}{q(2^q-1)(2^q+1)} +  \frac{2(q-1)}{3q}  >1/2,\\&\epsilon_{3q}\geq  \rho_{3q}(\PSigmaL_2(3^q))=  \frac{2\cdot 3^q}{q(3^q-1)(3^q+1)} +\frac{3(q-1)}{4q}>1/2,\\&\epsilon_q\geq \rho_q(\PGammaL_2(2^q)) =  \frac{1}{q(2^q-1)2^q(2^q+1)} + \frac{q-1}{6q}.
\end{align*}
\end{crl}

\begin{proof}
These groups satisfy the conditions of \autoref{Classical_Groups}(3), so using \autoref{PSL} and \autoref{PSL_even}, we can compute $\rho_{k}(\PSigmaL_2(p^q))$ for the given $k$'s.
\end{proof}

\begin{rmk}
In \autoref{Classical_Groups}(3), $\rho_{k}(G)$ cannot exceed $\max(\rho_{k}(H\rtimes K), \rho_{k}(N))$. Therefore, the only cases where \autoref{Classical_Groups} can be used to produce a meaningful example of an almost simple group with large $\rho_{k}^*$ are when $H$ is solvable. When $H$ is a ``subfield subgroup'', this happens when $H\cap \soc(G)\in \{\PSL_2(2),\PSL_2(3), \PSU_3(2)\}$. Other than those listed in \autoref{PSigmaL2}, we can also try subgroups of $\PGammaU_3(2^a)$ for $a$ relatively prime to $6$ to get a large $\rho_{4a}$ close to $8/9$ and $\rho_{6a}$ close to $3/4$. Although it is not ``classical'', subgroups of $\Aut(\Sz(2^{2a+1}))$ can be also used get $\rho_{4a}$ close to $4/5$. 

In the opposite direction, it might be possible to use \autoref{Classical_Groups} as a part of an exhaustive search of non-soluble monolithic groups to find the exact value of $\epsilon_k$. 
\end{rmk}

\section{Upper bounds and prime powers}

In this section, we find some upper bounds for $\rho_k(G)$ and $\epsilon_k$, focusing on the ratio $1/2$. It turns out that \autoref{conjecture} actually holds for all odd prime powers except for $3^1$. The methods developed in this section will be also used in the next section to study the case $k=4$, which is the only power of $2$ not covered by \autoref{PSigmaL2} and not previously known. 

We first give a simple inequality regarding normalisers of certain subgroups such as the Sylow subgroups, which will be our main tool in the verification of \autoref{conjecture} when $k$ is prime power.

\begin{lem}\label{self-normalising}
Let $k\in \mathbb{Z}_{>0}$ and suppose that a subset $S\subseteq G$ intersects every conjugacy class of elements of $G$ of order dividing $k$ (e.g. a Sylow $p$-subgroup when $k$ is a power of a prime $p$). Then \\{\upshape(1)} $$\rho_k(G)\leq \frac{\rho_k(\langle S\rangle)}{|\mathbf{N}_G(\langle S\rangle):\langle S\rangle|} - \frac{1}{|\mathbf{N}_G(\langle S\rangle)|} +\frac{1}{|G|}.$$ \\{\upshape(2)}If $\rho_k(G)>1/2$, then $\langle S\rangle$ is self-normalising in $G$, and if in addition $k=p^a$ for a prime $p$ and an integer $a>0$, then a Sylow $p$-subgroup of $G$ is self-normalising and more than half of its elements have order dividing $p^a$.
\end{lem}

\begin{proof}
(1) The union of $G$-conjugates of $S$ contains all elements of order dividing $k$ of $G$. Therefore, $$\rho_k(G)\leq \frac{\left|\bigcup_{x\in G}S^x\right|}{|G|} \leq \frac{\left(|\langle S\rangle |\rho_k(\langle S\rangle)-1\right)|G:\mathbf{N}_G(\langle S\rangle)|+1 }{|G|} \leq \frac{|\langle S\rangle|\rho_k(\langle S\rangle)}{|\mathbf{N}_G(\langle S\rangle)|} -\frac{1}{|\mathbf{N}_G(\langle S\rangle)|} + \frac{1}{|G|} .$$ 

(2) If $\rho_k(G)>1/2$, then part (1) forces $\langle S\rangle = \mathbf{N}_G(\langle S\rangle)$. If in addition $k=p^a$, then we may choose a Sylow $p$-subgroup $S$ of $G$. This $S$ satisfies the condition, so $\rho_k(S)\geq \rho_k(G)>1/2$.
\end{proof}

\begin{crl}\label{Sylow_quotient}
Suppose that $N\unlhd G$. Let $p$ be a prime, $f\in \mathbb{Z}_{>0}$, and $P/N$ a Sylow $p$-subgroup of $G/N$. Then $\rho_{p^f}(G) \leq \rho_{p^f}(P)$. 
\end{crl}

\begin{proof}
Suppose that $g\in G$ has order dividing $p^f$. Then some conjugate of $gN$ is in $P/N$, so some conjugate of $g$ is in $P$. Now we can apply \autoref{self-normalising}.
\end{proof}

Another part of our main strategy is the following. We partition a group into cosets of a certain normal subgroup. Since these cosets have equal sizes, the ratio $\rho_k$ in the entire group is equal to the average of the ratios computed in each coset. In particular, the ratio in the group cannot exceed the largest ratio among the cosets. This, combined with \autoref{monolithic_reduction}, gives the following reduction to the cosets of finite nonabelian simple groups in their automorphism groups.

\begin{lem}\label{Coset_bound}
Let $T$ be a finite nonabelian simple group. Suppose that $T^n\unlhd G\leq \Aut(T^n)=\Aut(T)\wr S_n=\Aut(T)^n\rtimes S_n$ for some $n\geq 1$. Then for any $k\in \mathbb{Z}_{>0}$, $$\rho_k(G)\leq \max_{\substack{1\leq r\leq n,\ (\ell_1, \dots, \ell_r) \text{ a partition of }n\\\forall j\in \{1,\dots,r\},\ \ell_j\text{ divides } k}} \prod_{j=1}^{r}\max_{x\in \Aut(T),\ o(x)\mid (k/\ell_j)}\rho_{k/\ell_j}(xT) \leq \max_{x\in \Aut(T),\ o(x)\mid k}\rho_k(xT),$$ where $xT$ denotes the coset of $T$ in $\Aut(T)$ containing $x$. In particular, $$\epsilon_k \leq \max_{\substack{T\text{ finite nonabelian simple}\\x\in \Aut(T),\  o(x)\mid k}}\rho_k(xT).$$
\end{lem}

\begin{proof}
Throughout this proof, for a subset $I$ of $\{1,\dots, n\}$ and a choice of an element $x_i\in \Aut(T)$ for each $i\in I$, we will denote by $(x_i)_{i\in I}$ the element $(y_1, \dots, y_n)\in \Aut(T)^n$ where $y_i=x_i\in \Aut(T)$ if $i\in I$ and $y_i = 1\in \Aut(T)$ for $i\notin I$.

Let $N=T^n\unlhd G$. Then the cosets of $N$ in $G$ partition $G$, so we get 
\begin{equation}\label{G_partition_into_N}
\rho_k(G) \leq \max_{g\in G}\rho_k(gN)
\end{equation}
where $gN$ denotes the coset of $N$ in $G$ containing $g$. 

Let $g = (f_1, \dots, f_n)\sigma$ for some $f_i\in \Aut(T)$ and $\sigma\in S_n$. We may write $g = \prod_{j=1}^{r}g_j$ with $g_j=(f_i)_{i\in I_j}\sigma_j$, where $I_1,\dots,I_r\subseteq \{1,\dots, n\}$ are the orbits of $\sigma$ and for each $1\leq j\leq r$, $\sigma_j\in S_n$ is the permutation which acts as $\sigma$ on $I_j$ and fixes $\{1,\dots,n\}\setminus I_j$ pointwise, so that $\sigma=\sigma_1\sigma_2\cdots\sigma_r$ is the disjoint cycle decomposition of $\sigma$. (Note that $\sigma_j$ is trivial when $|I_j|=1.$) Then for $a = (a_1,\dots, a_n)\in N$, the order of $ag$ is the least common multiple of the order of the $a_{I_j}g_j$'s, where $a_{I_j} = (a_i)_{i\in I_j}$. Let $\ell_j = |I_j|$ be the length of $\sigma_j$.

Note that $\left(a_{I_j}g_j\right)^{\ell} \notin \Aut(T)^n$ for $1\leq \ell < \ell_j$ and $$\left(a_{I_j}g_j\right)^{\ell_j} = \left(a_{\alpha}{f_\alpha}a_{\sigma\alpha}f_{\sigma\alpha}\cdots a_{\sigma^{\ell_j-1}\alpha} f_{\sigma^{\ell_j-1}\alpha}\right)_{\alpha \in I_j} \in \Aut(T)^n.$$ Moreover, the nonidentity components of the latter are conjugate to each other, since if $\alpha_j\in I_j$ then the components are \begin{align*}&a_{\sigma^m\alpha_j}f_{\sigma^m\alpha_j}\cdots a_{\sigma^{\ell_j-1}\alpha_j}f_{\sigma^{\ell_j-1}\alpha_j}a_{\alpha_j} f_{\alpha_j} \cdots a_{\sigma^{m-1}\alpha_j}f_{\sigma^{m-1}\alpha_j} \\=& \left(a_{\alpha_j}{f_{\alpha_j}}\cdots a_{\sigma^{\ell_j-1}\alpha_j} f_{\sigma^{\ell_j-1}\alpha_j}\right)^{a_{\alpha_j} f_{\alpha_j} \cdots a_{\sigma^{m-1}\alpha_j}f_{\sigma^{m-1}\alpha_j}}
\end{align*} for $m=0,\dots, \ell_j-1$. In particular, the order of $(a_{I_j}g_j)^{\ell_j}$ equals the order of any one of its nonidentity components. Therefore, for fixed $g_j$, $\alpha_j\in I_j$ and $a_{\sigma^\ell \alpha_j}$ for $1\leq \ell\leq \ell_j-1$, the number of choices of $a_{\alpha_j}\in T$ such that the order of $a_{I_j}g_j$ divides $k$ is $0$ if $\ell_j\nmid k$ and $$|T|\rho_{k/\ell_j}\left(\left({f_\alpha}a_{\sigma\alpha}f_{\sigma\alpha}\cdots a_{\sigma^{\ell_j-1}\alpha} f_{\sigma^{\ell_j-1}\alpha}\right)^{-1}T\right)\leq |T|\max_{x\in \Aut(T)}\rho_{k/\ell_j}(xT)$$ if $\ell_j\mid k$, so that \begin{equation}\label{Single_Coset_bound}\rho_k(gN) \leq \prod_{j=1}^{r}\max_{x\in \Aut(T)}\rho_{k/\ell_j}\left(xT\right).\end{equation} Together with \eqref{G_partition_into_N}, this implies $$\rho_k(G)\leq \max_{g\in G}\rho_k(gN)\leq \max_{\substack{(\ell_1, \dots, \ell_r) \text{ a partition of }n\\\forall j\in \{1,\dots, r\},\ \ell_j \mid k }} \prod_{j=1}^{r}\max_{x\in \Aut(T)}\rho_{k/\ell_j}(xT).$$ If $\rho_{k/\ell_j}(xT)>0$, then there exists some $y\in xT$ of order dividing $k/\ell_j$ and we may replace $x$ with $y$. Therefore it is enough to only consider $x\in \Aut(T)$ with $o(x)\mid k/\ell_j$ in the above inequality. The last assertion now follows from \autoref{monolithic_reduction}.
\end{proof}

Now we combine \autoref{self-normalising}, \autoref{Sylow_quotient} and \autoref{Coset_bound} to get some upper bounds for the powers of odd primes.

\begin{thm}\label{Prime_power}
{\upshape(1)} If $k$ is a power of a prime $p\geq 5$, then $\epsilon_k\leq 1/2$ and $\epsilon_k^* < 1/2$ if it exists. Consequently, \autoref{conjecture} holds for such $k$.\\ {\upshape(2)} Let $\alpha(m, a) = \max_{F\in \Gal(\mathbb{F}_{3^{3^a}}/\mathbb{F}_3)}\rho_{3^m}\left(F\cdot \PSL_2(3^{3^a})\right)$ for each $m\in \mathbb{Z}_{\geq 0}$.  If $k=3^f$ for some $f\in \mathbb{Z}_{>1                                                                                          }$, then $\epsilon_k\leq  \sup_{a\geq 1}\alpha(f,a)$.\\{\upshape(3)} If $G$ is a finite almost simple group, then $\rho_{3^f}(G)>1/2$ is possible only when $G=\PSigmaL(2,3^{3^a})$ with $1\leq a\leq f$.\\{\upshape(4)} If $k=3$ then $\epsilon_k\leq 1/2$.
\end{thm}

\begin{proof}
Let $G$ be a finite non-soluble group, and let $k$ be as described. If $k$ is a power of $3$, assume in addition that $G$ does not have any composition factor isomorphic to $\PSL_2(3^{3^a})$ for any $a\in \mathbb{Z}_{>0}$. By a result of Guralnick, Malle and Navarro \cite[Theorem 1.1]{GMN}, we know that $P$ cannot be self-normalizing in our situations. In particular, $|\mathbf{N}_G(P)|\geq 2|P|$, so by \autoref{self-normalising} we get $$\rho_k(G)\leq \frac{1}{2} - \frac{1}{\left|\mathbf{N}_G(P)\right|} + \frac{1}{|G|} .$$ Moreover, if $\rho_k(G)=1/2$, then we must have $|G| = \left|\mathbf{N}_G(P)\right| = 2|P|$, but that would imply that $P\lhd G$ and $|G:P|=2$, which is impossible since $G$ is non-soluble. Therefore, we get $\rho_k(G)<1/2$. If $k$ is a power of a prime $p\geq 5$, then this applies to all finite non-soluble $G$, so $\epsilon_k\leq 1/2$ and $\epsilon_k^*<1/2$ if it exists. This proves part (1).

Now suppose $k=3^f$ for some $f\in \mathbb{Z}_{>0}$. Suppose that $H$ is a finite non-soluble group with $\rho_k(H)\geq 1/2$ such that for every finite non-soluble group $K$ such that $|K|<|H|$, $\rho_k(H)>\rho_k(K)$; in other words, it is a ``minimal example'' of a finite non-soluble group with $\rho_k$ larger than or equal to certain number. By \autoref{monolithic_reduction}, $H$ must be monolithic with unique minimal normal subgroup $N$, and $H/N$ is soluble. The above argument shows that $N=\PSL_2(3^{3^a})^n$ for some $a, n\in \mathbb{Z}_{>0}$, so $N\unlhd H\leq\Aut(\PSL_2(3^{3^a})^n) = \PGammaL_2(3^{3^a})\wr S_n$. By \autoref{Sylow_quotient} applied to $H/N\leq \Aut(\PSL_2(3^{3^a})^n)/N \cong C_{2\cdot 3^a}\wr S_n$, we must have $H\leq \PSigmaL_2(3^{3^a})\wr R$, where $R$ is a Sylow $3$-subgroup of $S_n$. Also, $R$ should be transitive as a subgroup of $S_n$, since otherwise the product of the copies of $\PSL_2(3^{3^a})$ corresponding to one orbit of $R$ would be a normal subgroup of $H$ smaller than $N$. The structure of the Sylow subgroups of $S_n$ forces that $n= 3^v$ for some $v\in \mathbb{Z}_{\geq 0}$.

Let $P$ be a Sylow $3$-subgroup of $H$. By replacing $H$ with a conjugate in $\Aut(N)$, we may assume that $P\leq (U\rtimes \Phi)\wr R$, where $U$ is the subgroup of (upper) unitriangular matrices in $\PSL_2(3^{3^a})$, and $\Phi\cong \Gal(\mathbb{F}_{3^{3^a}}/\mathbb{F}_3)\cong C_{3^a}$ is the group of field automorphisms of $\PSL_2(3^{3^a})$. Note that $\diag(x,x^{-1}):=\overline{\begin{pmatrix}
x & 0 \\ 0 & x^{-1}
\end{pmatrix}}\in \PSL_2(3^{3^a})$ normalises $U$, and it commutes with $F \in \Phi$ if $x\in \mathbb{F}_{3^{3^a}}^\times$ is fixed by $F$. Moreover, this is not a $3$-element unless $x=x^{-1}$ which is equivalent to saying $x\in \mathbb{F}_3^\times$. If $H$ is contained in $\left(\PSL_2(3^{3^a})\rtimes \Phi\right)\wr R$, then for $x\in \mathbb{F}_{3^3}\setminus \mathbb{F}_3$, $d:=\left(\diag\left(x,x^{-1}\right), \dots, \diag\left(x,x^{-1}\right)\right) \in N$ normalises $(U\rtimes \Phi)^n$. Moreover, it commutes with $R$, so it normalises $P$. Since $d$ is not a $3$-element, $P$ is not self-normalising in $H$, so $\rho_k(H)<1/2$. Therefore $H$ is not contained in $\left(\PSL_2(3^{3^a})\rtimes \Gal(\mathbb{F}_{3^{3^a}}/\mathbb{F}_{3^3})\right)\wr R$. In particular, when $H$ is almost simple so that $n=1$, this forces $H$ to be $\PSigmaL(2,3^{3^a})$. In this case $H$ has the quotient $\Phi\cong C_{3^a}$, so by \autoref{quo}, $\rho_{3^f}(C_{3^a})\geq \rho_{3^f}(H)>1/2$, which forces $a\leq f$. This proves part (3).

Let $\alpha(m, a) = \max_{F\in \Gal(\mathbb{F}_{3^{3^a}}/\mathbb{F}_3)}\rho_{3^m}\left(F\cdot \PSL_2(3^{3^a})\right)$ for each $m\in \mathbb{Z}_{\geq 0}$. For $m<0$, define $\alpha(m,a)=0$. Note that $\alpha(m,a)\leq \alpha(m',a)\leq 1$ if $m\leq m'$. From \autoref{Coset_bound}, we get 
\begin{equation}\label{upper_bound_3f_n}
\rho_{3^f}(H) \leq \max_{\substack{r_0, \dots, r_v\\n=3^v = \sum_{i=0}^{v}r_i3^i}}\left(\prod_{m=0}^{v}\left(\alpha(f-m, a)\right)^{r_m} \right)\leq \alpha(f,a). 
\end{equation}

As we will see in \autoref{epsilon_9}, $\epsilon_{3^f}>1/2$ for $f>1$. Therefore there exists such $H$ for $k=3^f>3$, and clearly $\epsilon_{3^f} = \sup_{G\text{ finite non-soluble}}\rho_{3^f}(G) = \sup_{H\text{ as described above}}\rho_{3^f}(H)\leq \sup_{a\geq 1}\alpha(f,a)$, where the last inequality follows from \eqref{upper_bound_3f_n}. This proves part (2).

If $k=3$ so that $f=1$, then the right-hand side of \eqref{upper_bound_3f_n} becomes \begin{align}\label{alpha1a}
\alpha(1,a) &= \max_{F\in \Gal(\mathbb{F}_{3^{3^a}}/\mathbb{F}_3)}\left(\rho_3\left(F\cdot \PSL_2(3^{3^a})\right)\right)\nonumber\\&= \max\left(\rho_3\left(\PSL_2(3^{3^a})\right), \rho_3\left(F'\cdot \PSL_2(3^{3^a}\right)\right)\text{ where }F'\in \Gal(\mathbb{F}_{3^{3^a}}/\mathbb{F}_{3^{3^{a-1}}})\text{ has order }3\nonumber\\&=\max\left( \frac{3^{3^a}}{3^{2\cdot 3^a}-1}, \frac{\left| \{g\in \PSL_2(3^{3^a})\mid gg^{F'}g^{(F')^2}=1\}\right|}{|\PSL_2(3^{3^a})|} \right)
\end{align}
Suppose $A\in \SL_2(3^{3^a})$ satisfies $AA^{F'}A^{(F')^2}\in \mathbf{Z}(\SL_2(3^{3^a})) = \{I, -I\}$. If $A=\begin{pmatrix}
a_{11} & a_{12} \\ a_{21} & a_{22}
\end{pmatrix}$ (so $a_{11}a_{22}-a_{12}a_{21}=1$), then we get 
\begin{align}\label{AAFAFF}
&AA^{F'}A^{(F')^2} = \begin{pmatrix}
a_{11} & a_{12} \\ a_{21} & a_{22}
\end{pmatrix}\begin{pmatrix}
a_{11}^{F'} & a_{12}^{F'} \\ a_{21}^{F'} & a_{22}^{F'}
\end{pmatrix}\begin{pmatrix}
a_{11}^{(F')^2} & a_{12}^{(F')^2} \\ a_{21}^{(F')^2} & a_{22}^{(F')^2}
\end{pmatrix}\nonumber\\=& {\scriptsize\begin{pmatrix}
a_{11}a_{11}^{F'}a_{11}^{(F')^2} + a_{12}a_{21}^{F'}a_{11}^{(F')^2} + a_{11}a_{12}^{F'}a_{21}^{(F')^2} + a_{12}a_{21}^{F'}a_{21}^{(F')^2} &  a_{11}a_{11}^{F'}a_{12}^{(F')^2} + a_{12}a_{21}^{F'}a_{12}^{(F')^2} + a_{11}a_{12}^{F'}a_{22}^{(F')^2} + a_{12}a_{22}^{F'}a_{22}^{(F')^2}  \\ a_{21}a_{11}^{F'}a_{11}^{(F')^2} + a_{22}a_{21}^{F'}a_{11}^{(F')^2} + a_{21}a_{12}^{F'}a_{21}^{(F')^2} + a_{22}a_{22}^{F'}a_{21}^{(F')^2}  & a_{21}a_{11}^{F'}a_{12}^{(F')^2} + a_{22}a_{21}^{F'}a_{21}^{(F')^2} + a_{21}a_{12}^{F'}a_{22}^{(F')^2} + a_{22}a_{22}^{F'}a_{22}^{(F')^2}
\end{pmatrix}}\nonumber\\\in& \{I, -I\}.
\end{align}
Let $x\in \mathbb{F}_{3^{3^{a-1}}}^\times$, so that $x^{F'}=x$. Then $X:=\begin{pmatrix}
x & 0 \\ 0 & x^{-1}
\end{pmatrix}\in \SL_2(3^{3^a})$. Suppose that $(XA)(XA)^{F'}(XA)^{(F')^2}\in \{I,-I\}$. Then $XA = \begin{pmatrix}
xa_{11} & xa_{12} \\ x^{-1}a_{21} & x^{-1}a_{22}
\end{pmatrix}$, and the above formula shows that 
\begin{align}\label{XA}
&a_{11}a_{11}^{F'}a_{12}^{(F')^2} + a_{12}a_{21}^{F'}a_{12}^{(F')^2} + a_{11}a_{12}^{F'}a_{22}^{(F')^2} + a_{12}a_{22}^{F'}a_{22}^{(F')^2}\nonumber \\=& x^3a_{11}a_{11}^{F'}a_{12}^{(F')^2} + xa_{12}a_{21}^{F'}a_{12}^{(F')^2} + xa_{11}a_{12}^{F'}a_{22}^{(F')^2} + x^{-1}a_{12}a_{22}^{F'}a_{22}^{(F')^2} = 0.
\end{align}
It follows that \begin{align*}
(x^2-1)a_{11}a_{11}^{F'}a_{12}^{(F')^2} + (x^{-2}-1)a_{12}a_{22}^{F'}a_{22}^{(F')^2} = 0.
\end{align*}
If $x\neq \pm 1$, then this becomes 
\begin{equation}\label{xsquare}
x^2a_{11}a_{11}^{F'}a_{12}^{(F')^2} =- a_{12}a_{22}^{F'}a_{22}^{(F')^2}.
\end{equation}
Therefore, when $a_{11}, a_{12}, a_{22}$ are nonzero, there are at most $4$ choices (including $\pm 1$) of $x\in \mathbb{F}_{3^{3^{a-1}}}^\times$ such that $(XA)(XA)^{F'}(XA)^{(F')^2}\in \{I,-I\}$. Similarly, when $a_{11},a_{21},a_{22}$ are nonzero, there are at most $4$ choices of such $x$. 

If $a_{11}=0$, then $a_{12}a_{21}=-1$. Moreover, if $x\neq \pm 1$, then from \eqref{xsquare} we get $a_{22}=0$, which contradicts \eqref{XA}. Therefore $x =\pm 1$, so we only have $2$ choices of $x$ here. Similarly if $a_{22}=0$ we have $x=\pm 1$. If $a_{21}=0$, then $a_{11}a_{22}=1$, and from \eqref{AAFAFF}, we get $a_{11}a_{11}^{F'}a_{11}^{(F')^2}= \pm 1$ and $x^3a_{11}a_{11}^{F'}a_{11}^{(F')^2}= \pm 1$. In particular $x^3=\pm 1$, so $x=\pm 1$. Similarly $a_{12}=0$ forces $x=\pm 1$. Therefore, in any coset of the subgroup $\left\lbrace\begin{pmatrix}
x & 0 \\ 0 & x^{-1}
\end{pmatrix}\mid x\in \mathbb{F}_{3^{3^{a-1}}}^\times\right\rbrace$ of order $3^{3^{a-1}}-1$ in $\SL_2(3^{3^a})$, there are at most $4$ elements $A$ such that $AA^{F'}A^{(F')^2}\in \mathbf{Z}\left(\SL_2(3^{3^a})\right)$. Therefore,  
$$\frac{\left| \{g\in \PSL_2(3^{3^a})\mid gg^{F'}g^{(F')^2}=1\}\right|}{|\PSL_2(3^{3^a})|} \leq \frac{4}{3^{3^{a-1}}-1}. $$

By \eqref{alpha1a}, we get $\alpha(1,a)\leq \frac{4}{3^{3^{a-1}}-1}$. Therefore if $a>1$, then $\alpha(1,a)\leq 2/13$. Finally, when $a=1$, we can manually compute $$\frac{\left| \{g\in \PSL_2(3^{3^a})\mid gg^{F'}g^{(F')^2}=1\}\right|}{|\PSL_2(3^{3^a})|} = 1/12.$$ Therefore, by \eqref{upper_bound_3f_n}, the group $H$ we described earlier with $\rho_3(H)>1/2$ cannot exist for $f=1$, which shows that $\epsilon_3\leq 1/2$. 
\end{proof}

\begin{prop}\label{epsilon_9}
The conjecture fails for all multiples of $9$. More precisely, $\epsilon_9 \geq 191/364 = \rho_9(\PSigmaL_2(3^3))>1/2$, and this is the largest ratio among the finite almost simple groups. 
\end{prop}

\begin{proof}
By \autoref{Prime_power} and its proof, if $G$ is a finite almost simple group, then $\rho_9(G)\geq 1/2$ can happen only when $G=\PSigmaL_2(3^3)$ or $G=\PSigmaL_2(3^9)$. The latter gives $\rho_9(\PSigmaL_2(3^9))=2472421763637/34315188593868 = 0.072\dots$, and the former gives $\rho_9(\PSigmaL_2(3^3)) = 191/364 >1/2$. Therefore this is the largest $\rho_9$ among the almost simple groups.
\end{proof}

\section{The exact value of $\epsilon_4$}

So far, we proved that the conjecture holds for powers of all primes $\geq 5$, and also for $2$ and $3$, while it fails for higher powers of $2$ and $3$ except for $4$. In this section, we compute the value of $\epsilon_4$. Note that the normaliser part of \autoref{self-normalising} in this situation is not as powerful as in \autoref{Prime_power}, since unlike odd primes, there are many finite non-soluble groups with self-normalising Sylow 2-subgroups. On the other hand, when we do have a Sylow 2-subgroup which is not self-normalizing, it gives a better upper bound of $1/3$, since the index of the Sylow $2$-subgroup in the normaliser is not divisible by $2$ in this case.

First we reduce to almost simple groups. We will need the following lemma which can be easily derived from the work of Liebeck and MacHale and its extension by Potter:

\begin{lem}\label{k=2_Coset}
{\upshape(1)} {\upshape{\cite[Theorem 4.1]{LM}, \cite[Theorem 3.3]{Potter}}} If $G$ is a finite group and $x\in \Aut(G)$ inverts more than half of the elements of $G$, then $G$ has an abelian normal subgroup $A$ and $G/A$ is an elementary abelian $2$-group. If $x$ inverts more than $4/15$ of the elements of $G$ then $G$ is soluble.\\{\upshape(2)} If $T$ is a finite nonabelian simple group and $x\in \Aut(T)$, then $\rho_2(xT)\leq 4/15$.
\end{lem}

\begin{proof}
For (2), suppose that $\rho_2(xT)>0$. We may assume that $o(x)=2$. Note that for each $g\in T$, $o(xg)=2$ if and only if $g^x = g^{-1}$. Therefore $\rho_2(xT)$ is the ratio of elements of $T$ inverted by $x$. Now part (1) shows that this ratio cannot exceed $4/15$ since $T$ is non-soluble.
\end{proof}

Now we can reduce the conjecture for $k=4$ to almost simple groups.

\begin{thm}\label{almost_simple_reduction}
Let $G$ be a finite non-soluble group such that $\rho_4(G)>\rho_4(K)$ for all finite non-soluble groups $K$ with $|K|<|G|$. If $\rho_4(G)>7/15$, then $G$ is almost simple.
\end{thm}

\begin{proof}
As in the proof of \autoref{Prime_power}, $G$ must be a monolithic group with unique minimal normal subgroup $N:=T^n$ for some $n\in \mathbb{Z}_{>0}$ and a finite nonabelian simple group $T$, and $G/N$ is soluble. Also, $N\unlhd G\leq  \Aut(T)\wr P = \Aut(T)^n\rtimes P$ for some transitive $2$-subgroup $P$ of $S_n$, so $n=2^f$ for some $f\in \mathbb{Z}_{\geq 0}$. 

For each coset $gN$ of $N$ in $G$, let $s(gN)$ be the element of $P<S_n$ such that $g = (a_1,\dots,a_n)s(gN) \in \Aut(T)^n\rtimes P$ for some $(a_1,\dots,a_n)\in \Aut(T)^n$. Let $c_1$ be the number of cosets $gN$ such that $s(gN)=1\in S_n$, $c_2$ be the number of those whose $s(gN)$ is a $2$-cycle, and $c_3 = |G:N| - c_1-c_2$ be the number of other cosets, which must have either a $4$-cycle or at least two $2$-cycles in the disjoint cycle decomposition of $s(gN)$. 

Note that if $s(gN)=1$ then $s(ghN)=s(gN)s(hN)=s(hN)$, so $c_2$, $c_3$ are multiples of $c_1$, and $c_2/c_1$, $c_3/c_1$ are the number of elements of the transitive $2$-subgroup $P=\{s(gN)\mid g\in G\}\leq S_n$. In particular $c_1+c_2+c_3 =|P|c_1\geq nc_1$, where the last inequality follows from the transitivity of $P$. 

By \eqref{Single_Coset_bound}, \autoref{k=2_Coset} and the fact that $\rho_1(xT)\leq 1/60$ for all $x\in \Aut(T)$, we get 
{\small
\begin{align}
\rho_4(G) &= \frac{1}{|G:N|}\sum_{gN \in G/N}\rho_4(gN)\nonumber\\&= \frac{1}{|G:N|}\left(\sum_{\substack{gN\in G/N\\s(gN)=1}}\rho_4(gN) + \sum_{\substack{gN\in G/N\\s(gN)\text{ is a 2-cycle}}}\rho_4(gN) + \sum_{\substack{gN\in G/N\\1\neq s(gN)\text{ is not a 2-cycle}}}\rho_4(gN)\right)\nonumber\\&\leq \frac{1}{|G:N|}\left(c_1\left(\max_{x\in \Aut(T)}\rho_4(xT) \right)^n + c_2\frac{4}{15}\left(\max_{x\in \Aut(T)}\rho_4(xT) \right)^{n-2} + c_3\frac{4^2}{15^2} \right)\nonumber\\&\leq \frac{1}{c_1+c_2+c_3}\left(c_1 + \frac{4c_2}{15} + \frac{16c_3}{225} \right) = \frac{225c_1 + 60c_2 + 16c_3}{225(c_1+c_2+c_3)}\label{c1c2c3}.
\end{align} }
If $n\geq 4$, then $c_2/c_1 + c_3/c_1\geq n-1 \geq 3$, so \eqref{c1c2c3} becomes
\begin{align*}
&\frac{90c_1 +135c_1 + 60c_2+16c_3   }{225(c_1+c_2+c_3)}\leq \frac{90c_1+45(c_2+c_3) + 60c_2+16c_3   }{225(c_1+c_2+c_3)}\\=& \frac{90c_1 + 105c_2 + 61c_3 }{225(c_1+c_2+c_3)} < \frac{105(c_1+c_2+c_3)}{225(c_1+c_2+c_3)}=\frac{7}{15}.
\end{align*}

If $n=2$, then the transitivity forces $P=C_2$ and $G\not\leq \Aut(T)^2$. Consider the subgroup $H = G\cap \Aut(T)^2$. Note that $H$ is the union of the cosets $gN$ of $G$ such that $s(gN)=1$, and $G\setminus H$ is the union of all other cosets $gN$, whose $s(gN)$ is a single $2$-cycle since this is the only nonidentity element of $S_2$. Then by the arguments above and \eqref{Single_Coset_bound}, $\rho_4(G\setminus H)\leq 4/15$ and $7/15<\rho_4(G) = \left(\rho_4(H) + \rho_4(G\setminus H)\right)/2 $, so we get $$\rho_4(H) > 14/15 - \rho_4(G\setminus H) > 14/15 - 4/15 = 2/3 > \rho_4(G).$$ This contradicts the minimality of $G$.
\end{proof}

We now check the conjecture for each finite almost simple group. Our main strategy is the following. Let $G$ be a minimal counterexample. By the minimality of $G$, any subgroup satisfying the condition of \autoref{self-normalising} must be soluble, self-normalizing and have $\rho_4\geq 7/15$. For most of the finite almost simple group, we locate a subgroup satisfying the condition of \autoref{self-normalising}, but which does not satisfy one of the three conditions stated above. 

\begin{prop}\label{4_sporad}
Let $G$ be a finite almost simple group whose simple factor is a sporadic simple group or an alternating group. Then $\rho_4(G)\leq 7/15$. 
\end{prop}

\begin{proof}
Since $|\Out(T)|\leq 2$ for sporadic simple groups $T$, the number of finite almost simple groups $G$ whose nonabelian simple factor is sporadic is not too large. Using the information of sporadic simple groups available in \cite{Atlas, WebAtlas}, we computed $\rho_4$ of all such $G$, and listed them in \autoref{rho4table}. 

If $\soc(G)$ is $A_n$ for some $n\geq 5$, then $G$ is either $A_n$ or $S_n$ except when $n=6$. We may write $n = 4m+2s + t$ for some $m\in \mathbb{Z}_{>0}$ and $s,t\in \{0,1\}$. For $n<12$, $G$ is one of the following groups: $A_n$ and $S_n$ for $5\leq n\leq 11$, $\text{Mathieu group } M_{10}, \PGL_{2}(9)$, and $\Aut(S_6)$. The $\rho_4$ for these groups are computed in \autoref{alt_table} for $n=7,9,10,11$, and in \autoref{4equichartable} for $n=5,6,8$. Suppose $n\geq 12$ so that $m\geq 2$, and consider the following subgroup: $$H = \langle (1\ 2\ 3\ 4), (1\ 3), (5\ 6\ 7\ 8), (5\ 7), \dots, (4m-3\ 4m-2\ 4m-1\ 4m), (4m-3\ 4m-1), \sigma \rangle < S_n$$ where $\sigma=1$ if $s=0$ and $\sigma = (4m+1\ 4m+2)$ if $s=1$. Then every element of $S_n$ of order dividing $4$ is conjugate to an element of $H$, and every element of $A_n$ of order dividing $4$ is conjugate in $A_n$ to an element of $H\cap A_n$. Also, the element $(1\ 5)(2\ 6)(3\ 7)(4\ 8)\in A_n\setminus H$ normalises both $H$ and $H\cap A_n$, so $\mathbf{N}_{S_n}(H) > H$ and $\mathbf{N}_{A_n}(H\cap A_n)>H\cap A_n$. Similarly $(1\ 9)(2\ 10)(3\ 11)(4\ 12)$ normalises $H$ and $H\cap A_n$, so in fact $|\mathbf{N}_{S_n}(H):H|\geq 4$ and $|\mathbf{N}_{A_n}(H):H\cap A_n|\geq 4$. By \autoref{self-normalising}, $\rho_4(S_n)\leq 1/4$ and $\rho_4(A_n)\leq 1/4$.
\end{proof}

\begin{table}[h]
\begin{center}
\caption{$\rho_4(G)$ for almost simple $G$ with $\soc(G)=A_n$, $n=7,9,10,11$}\label{alt_table}
\begin{tabular}{|c|c|c|}
\hline
$n$ & $\rho_4(A_n)$ & $\rho_4(S_n)$\\
\hline
$7$ & $92/315 = 0.292\dots$ & $ 67/315 =0.212\dots $\\
\hline
$9$ & $ 316/2835=0.111\dots$ & $2101/22680 = 0.092\dots$\\
\hline
$10$ & $ 6241/113400=0.055\dots $ & $ 6833/113400=0.060\dots$\\
\hline
$11$ & $29731/1247400=0.023\dots $ & $20729/623700=0.033\dots $\\
\hline
\end{tabular}
\end{center}
\end{table}

To study the almost simple groups of Lie type, we will use the following elementary lemma.

\begin{lem}\label{normalise}
Let $p$ be a prime. Suppose that $T\unlhd G$, $H$ is a subgroup of $T$ containing a Sylow $p$-subgroup of $T$, and a subgroup $K\leq G$ stabilizes the $T$-conjugacy class of $H$. If $S$ is a $p$-subgroup of $K$, then $S$ normalises some conjugate of $H$. In particular, if $G=TK$ then $\mathbf{N}_G(H)$ contains a Sylow $p$-subgroup of $T$.
\end{lem}

\begin{proof}
Note that $|T|_p=|H|_p$, so the number of $T$-conjugates of $H$ is not divisible by $p$. By the orbit-stabilizer theorem, the action of $S$ on the $T$-conjugacy class of $H$ must have an orbit of length $1$.
\end{proof}

\begin{thm}\label{4_equichar}
Let $G$ be a finite almost simple group whose simple factor is of Lie type in even characteristic. Then $\rho_4(G)\leq 7/15$. 
\end{thm}

\begin{proof}
Let $G$ be an almost simple group such that $T=\soc(G)$ is a simple group of Lie type over $\mathbb{F}_q$ where $q=2^e$ for some $e\in\mathbb{Z}_{>0}$. Suppose that $G$ is a minimal counterexample, so that $\rho_4(G)>7/15$ and that $\rho_4(G)>\rho_4(K)$ for all non-soluble $K$ with $|K|<|G|$. By \autoref{Sylow_quotient} we may assume that $G = TP$ for some Sylow $2$-subgroup $P$ of $G$, and that if $Q=T\cap P$, then $\rho_4(P/Q)=\rho_4(G/T)\geq \rho_4(G)>7/15$. Recall the following facts about automorphisms of $T$, cf. \cite[Section 2.5]{GLS}: An element of $\Aut(T)$ can be written as $idfg$ for $i\in T$, $d$ a diagonal automorphism, $f$ a field automorphism and $g$ a graph automorphism. $\Aut(T)=InnDiag(T)\rtimes (\Phi_T\Gamma_T)$ for the group $InnDiag(T)$ of inner and diagonal automorphisms and a group of (standard) field and graph automorphisms $\Phi_T\Gamma_T$. Also, we can choose a group of (standard) field automorphisms $\Phi_T$ so that $\Phi_T\Gamma_T=\Phi_T\times \Gamma_T$ for a group of (standard) graph automorphisms $\Gamma_T$ if $T\cong \PSL_n(q), \POmega_{2n}^+(q), E_6(q)$, and in other cases $\Phi_T\Gamma_T$ is cyclic and $\Phi_T$ has index $1$ or $2$ in $\Phi_T\Gamma_T$. $\Gamma_T$ is either trivial or $C_2$, except when $T\cong \POmega_8^+(q)$ in which case $\Gamma_T\cong S_3$.

Since $|InnDiag(T)/T|$ is relatively prime to the characteristic, $|\Aut(T)|_2 = |T\rtimes \Phi_T\Gamma_T|_2$. Therefore, by replacing $G$ with a conjugate in $\Aut(T)$, we may assume that $P\leq T\rtimes \Phi_T\Gamma_T$, so that the elements of $G$ does not involve any nontrivial diagonal automorphism. In particular, $G\cap InnDiag(T)=T$ and $G/T\cong G\cap \Phi_T\Gamma_T$. In this situation, an element of $G$ can be written as $ifg$ as above. 

Since $\Phi_T\Gamma_T$ is either cyclic or $\Phi_T\times\Gamma_T$, the elements of order dividing $4$ form a subgroup $K$ in $G/T$. Let $\tilde{K}$ be the inverse image of $K$ in $G$. Then every element of $G\setminus \tilde{K}$ has order not dividing $4$, so $\tilde{K}$ satisfies the condition of \autoref{self-normalising}. By the minimality of $G$, we must have $\tilde{K}=G$, so that every element of $G/T$ must have order dividing $4$. In other words, $G\cap \Phi_T\Gamma_T$ has exponent dividing $4$. Hence we may assume that $G\cap \Phi_T\Gamma_T \leq P$.

Following \cite[Section 4]{GMN}, let $R$ be any maximal parabolic subgroup of $T$ containing $Q$ which is normalised by $(\Phi_T\Gamma_T)\cap G$, which exists except when $T$ is one of the following: $\PSL_n(q)$ with odd $n$, $\PSp_4(q)$ and $F_4(q)$, cf. \cite[Section 2.6]{GLS} (see also \cite{KL, BHR, Cr}). Then $\mathbf{N}_G(R)$ contains $P$, so $\rho_4(G)\leq \rho_4(\mathbf{N}_G(R))$ by \autoref{self-normalising}. Since $\mathbf{N}_G(R)$ is proper in $G$, by the minimality of $G$, $\mathbf{N}_G(R)$ must be soluble. By \cite[Lemma 5.4]{Bur}, $T$ must be one of the following (including all exceptions we made above): 
\begin{itemize}
\item $\PSL_n(q)$ with $n$ odd. We may assume that $Q$ is the upper unitriangular subgroup of $\PSL_n(q)$, $\Phi_T$ is the standard field automorphisms acting as entry-wise Galois automorphisms, and that the nontrivial graph automorphism $g\in \Gamma_T$ maps a matrix $X$ to $A(X^T)^{-1}A^{-1}$ where $A$ is the matrix $$\begin{pmatrix}
& & & & -1 \\ & & & 1 & \\ & & -1 & & \\ & \Ddots  & & & \\ (-1)^{n} & & & &
\end{pmatrix}$$ as described in \cite[Section 2.7]{GLS}. Then $(\Phi_T\Gamma_T)\cap G$ normalises the parabolic subgroup $S$ which stabilizes a flag of dimensions $(1, n-1,n)$, i.e., the block upper triangular subgroup of block sizes $1$, $n-2$, $1$. Therefore, if $n\geq 5$, then $\mathbf{N}_G(S)$ for this $S$ is a non-soluble proper subgroup of $G$ containing $P$. By \autoref{self-normalising}, this contradicts the minimality of $G$. Therefore $n$ must be $3$.

Suppose that $q\geq 2^5$. Then there exists an element $a\in \mathbb{F}_q\setminus \mathbb{F}_2$ which is fixed by all elements of $\Phi_T$ of order dividing $4$. The image $x\in \PSL_3(q)$ of the matrix $$\begin{pmatrix}
a & & \\ & 1 & \\ & & a^{-1}
\end{pmatrix}$$ is a nonidentity element that commutes with all elements of $(\Phi_T\Gamma_T)\cap G$ and also normalises $Q$. Therefore $x$ normalises $P$ and $x\notin P$, so $P$ is not self-normalizing. Since $P$ is a Sylow $2$-subgroup of $G$, $|\mathbf{N}_G(P):P|$ is not divisible by $2$, so it follows that $|\mathbf{N}_G(P):P|\geq 3$. By \autoref{self-normalising} and the assumption $\rho_4(G)>13/30$, this is impossible, so we get $q\leq 2^4$. Therefore $G$ must be one of the following: $\PSL_3(2), \PSL_3(4), \PSL_3(8), \PSL_3(16),$ and their semidirect products with subgroups of $\Gal(\mathbb{F}_q/\mathbb{F}_2)\times \Gamma_T$. 
\item $\PSL_2(q)$. Here we have $\Gamma_T=1$. We may again assume that $\Phi_T$ consists of the standard field automorphisms acting by entry-wise Galois automorphisms, and that $Q$ is the upper unitriangular subgroup of $T$ and $P \leq Q\rtimes \Phi_T$. The upper triangular (Borel) subgroup $B<\PSL_2(q)$ normalises $Q$ and is normalised by $\Phi_T$, so $\mathbf{N}_G(Q)$ is a proper subgroup of $G$ which contains both $B$ and $P$. Unless $q\leq 16$, the Galois automorphisms of order dividing $4$ fixes some element $a\in \mathbb{F}_q\setminus \mathbb{F}_2$, so that the image of $$\begin{pmatrix}
a & 0 \\ 0 & a^{-1}
\end{pmatrix}$$ in $\PSL_n(q)$ is a nonidentity element that normalises $P$. Again by \autoref{self-normalising} this is impossible. Therefore $q\leq 16$. 
\item $\PSL_n(2)$ with $n\leq 5$. There are $6$ possible $G$ in this case: $\PSL_3(2)$, $\PGL_2(7)$, $\PSL_4(2)\cong A_8$, $S_8$, $\PSL_5(2)$, $\Aut(\PSL_5(2))$. 
\item $\PSU_3(q)$ with $q\geq 4$. In this case $\Gamma_T=1$, so if there is an element $x\in \mathbf{N}_T(Q)\setminus Q$ with entries fixed by the elements of $\Phi_T$ of order dividing $4$, then $x$ normalises $P$ so $P$ is not self-normalizing, which forces $|\mathbf{N}_G(P):P|\geq 3$. We may assume that $Q$ is upper triangular; then any non-central diagonal matrix with entries fixed by all elements of $\Phi_T$ of order dividing 4 can be used as $x$. Therefore, to have a self-normalizing $P$, we must have $q\leq 4$. Since $q\geq 4$, we get $q=4$ and $G=\PSU_3(4)$ or $\PSigmaU_3(4)$. 
\item $\PSU_n(2)$ with $n=4, 5$. There are $4$ possible $G$ in this case: $\PSU_4(2)$, $\PSigmaU_4(2)$, $\PSU_5(2)$, $\PSigmaU_5(2)$. 
\item $\PSp_4(q)$. Recall that we defined $e=\log_2(q)$. Then $\Phi_T\Gamma_T\cong C_{2e}$ and $|\Phi_T\Gamma_T:\Phi_T|=2$. Since the exponent of $G\cap \Phi_T\Gamma_T$ divides $4$, we must have $G/T\cong (G\cap \Phi_T\Gamma_T)\cong C_1, C_2$ or $C_4$. 

By \cite[Proposition 7.2.5, Table 8.14]{BHR}, $P$ is contained in a maximal subgroup $M$ of $G$ such that either $M$ has $\GL_2(q)$ as a composition factor, or $G\not< T\rtimes \Phi_T$ and $M= \mathbf{N}_G(B)$ for the Borel subgroup $B=\mathbf{N}_T(Q)\cong Q\rtimes C_{q-1}^2$ of $T$. By the minimality of $G$, $M$ must be soluble, so the former cannot happen unless $q=2$. 

Suppose that $q>2$ so that the latter happens. By \autoref{self-normalising}, $\rho_4(M)\geq \rho_4(G)>7/15$. Note that $MT=G$, so $|M:B|=|G:T|=|P:Q|$, hence $M=BP=(Q\rtimes C_{q-1}^2)\rtimes (G\cap \Phi_T\Gamma_T)$. Since $Q$ is normal in $M$, by \autoref{quo}, $\rho_4(M/Q)\geq \rho_4(M)>7/15$. According to \cite[Proposition 7.2.5]{BHR}, $G\cap \Phi_T\Gamma_T$ in fact normalises $C_{q-1}^2$, so $M/Q\cong C_{q-1}^2\rtimes (G\cap \Phi_T\Gamma_T)$. The action is described in \cite[Lemma 7.2.2]{BHR}; for convenience we describe it here for a generator $\gamma$ of $G\cap \Phi_T\Gamma_T$. Recall that $e=\log_2(q)$. We may write $e=2n+1$ if it is odd, and $e=4n+2$ if it is even. Then for $(x,y)\in C_{q-1}^2$, $$\gamma(x,y)\mapsto (x^{2^n}y^{2^n}, x^{2^n}y^{-2^n}).$$ Therefore, \begin{align}\label{4thpower}((x,y), \gamma)^4=& ((x^{2^n+1}y^{2^n}, x^{2^n}y^{1-2^n}), \gamma^2)^2\nonumber\\=& ((x^{(2^n+1)(1+2^{2n+1})}y^{2^n(1+2^{2n+1})}, x^{2^n(1+2^{2n+1})}y^{(1-2^n)(1+2^{2n+1})}), 1).
\end{align}
If $e$ is odd, then $x^{2^{2n+1}}=x^q=x$ for all $x\in C_{q-1}$, so the above expression becomes $$ ((x^{2^{n+1}+2}y^{2^{n+1}}, x^{2^{n+1}}y^{2-2^{n+1}}), 1).$$ This is the identity if and only if $x^{2^{n+1}+2}y^{2^{n+1}}=x^{2^{n+1}}y^{2-2^{n+1}}=1$, which implies $x^2 = y^{2-2^{n+2}}$. Since $q-1=2^{2n+1}-1$ is relatively prime to both $2$ and $2-2^{n+2}$, for each $x\in C_{q-1}$ there exists at most one $y\in C_{q-1}$ such that $((x,y),\gamma)^4=1$. Therefore, $\rho_4((B/Q)\gamma)\leq 1/(q-1)$ and $$\rho_4(M/Q)=\rho_4((B/Q)\sqcup (B/Q)\gamma) = \frac{\rho_4(B)+\rho_4((B/Q)\gamma)}{2} \leq \frac{1/(q-1)^2 + 1/(q-1)}{2}$$ which is less than $7/15$ when $q> 3$.  

If $e$ is even, then \eqref{4thpower} gives the identity if and only if $$x^{(2^n+1)(1+2^{2n+1})}y^{2^n(1+2^{2n+1})}= x^{2^n(1+2^{2n+1})}y^{(1-2^n)(1+2^{2n+1})}=1$$ which implies $x^{1+2^{2n+1}} = y^{(1-2^{n+1})(1+2^{2n+1})}$. Therefore for each $y\in C_{q-1}$, there exists at most $2^{2n+1}+1$ elements $x\in C_{q-1}$ such that $((x,y),\gamma)^4=1$. Therefore, $\rho_4(B\gamma)\leq (2^{2n+1}+1)/(q-1) = 1/(2^{2n+1}-1)$. The same bound applies to $B\gamma^3$, so \begin{align*}
\rho_4(M/Q) \leq& \frac{1}{4}\left(\rho_4(B/Q) + \rho_4((B/Q)\gamma) + \rho_4((B/Q)\gamma^2)+\rho_4((B/Q)\gamma^3)\right)\\\leq&\frac{1}{4(q-1)^2} + \frac{1}{2(2^{2n+1}-1)}+\frac{1}{4}  
\end{align*}
which is less than $7/15$ when $q>4$. Therefore, the only remaining cases are $q=2$ and $4$. 
\item $\PSp_6(2)$. In this case $T=\Aut(T)$, so $G=\PSp_6(2)$. 
\item $\POmega_8^+(2)$. In this case, there are no diagonal or field automorphisms, and as mentioned above, at most one nontrivial graph automorphism can appear and it must have order $2$. Therefore $G$ is either $\POmega_8^+(2)$ or $\PSO_8^+(2)$. 
\item $F_4(q)$. There is a subgroup $H=[q^{20}].\Sp_4(q).(q-1)^2$ in $F_4(q)$, whose conjugacy class is stabilized by $\Phi_T\Gamma_T$, cf. \cite[Table 8]{Cr}. Here $[q^{20}]$ denotes some group of order $q^{20}$. Since $|H|_2=|F_4(q)|_2$, by \autoref{normalise}, $\mathbf{N}_G(H)$ is a proper subgroup of $G$ that contains a Sylow $2$-subgroup of $G$. By \autoref{self-normalising}, $\rho_4(G)\leq \rho_4(\mathbf{N}_G(H))$. Therefore, by the minimality of $G$, $\mathbf{N}_G(H)$ must be soluble. Since $H$ has a composition factor $\Sp_4(q)$, $q$ must be $2$. 
\item $G_2(2)'\cong \PSU_3(3)$. In this case, $G$ is either $G_2(2)'\cong \PSU_3(3)$ or $G_2(2)\cong\PSigmaU_3(3)$. 
\item $\vphantom{a}^2B_2(q)$ with $q\geq 8$. In this case, $q$ is an odd power of $2$, so the elements of $G$ do not involve any field automorphism. Since $\Aut(\vphantom{a}^2B_2(q))$ has no diagonal and graph automorphism, the only possibility is $G=T\cong \vphantom{a}^2B_2(q)$. This was already studied in \autoref{Suzuki}: $\rho_4(\vphantom{a}^2B_2(q)) = q^4/(q^2(q-1)(q^2+1)) < 1/(q-1)<7/15$, so this is impossible.
\item $\vphantom{a}^2F_4(2)'$. In this case, $G$ is either $\vphantom{a}^2F_4(2)'$ or $\vphantom{a}^2F_4(2)$. 
\end{itemize}
For each of the possible $G$ found above, we computed $\rho_4(G)$ in \autoref{4equichartable}. Some of them were computed only for the largest subgroup in $TS$ for a Sylow $2$-subgroup of $\Aut(T)$. Since $G=TP\leq TS$, we have $\rho_4(G) \leq |TS:G|\rho_4(TS)\leq |TS:T|\rho_4(TS)$, so if $\rho_4(TS)\leq 7/(15|TS:T|)$ then $\rho_4(G)\leq 7/15$. From the table we can see that $\rho_4(G)>7/15$ only happens when $G \cong S_5$ or $M_{10}$, and $\rho_4(G)>7/15$ never happens.
\end{proof}

\begin{table}[h]
\begin{center}
\caption{$\rho_4(G)$ for $G$ of Lie type in even characteristic in the proof of \autoref{4_equichar}.}\label{4equichartable}
{\small
\begin{tabular}{|c|c|c|}
\hline
$T$ & $G$ & $\rho_4(G)$\\
\hline
\multirow{9}{*}{$\PSL_n(q)$, $n$ odd} & $\Aut(\PSL_3(16))$ & $\frac{112601}{26732160}=0.0042\dots < \frac{7}{15|\Out(\PSL_3(16))|}=7/60$ \\
\cline{2-3}
& $\Aut(\PSL_3(8))$ & $ \frac{4607}{1545264}=0.0029\dots < \frac{7}{15|\Out(\PSL_3(8))|} = 7/90$\\
\cline{2-3}
& $\PSL_3(4)\rtimes \Phi_T\Gamma_T$ & $617/5040=0.122\dots$\\
\cline{2-3}
& Intermediate groups & $421/2520=0.167\dots $\\
\cline{3-3}
& between $\PSL_3(4)$ and& $277/2520=0.109\dots $\\
\cline{3-3}
&  $\PSL_3(4)\rtimes \Phi_T\Gamma_T$  & $431/2520=0.171\dots$\\
\cline{2-3}
& $\PSL_3(4)$ & $64/315=0.203\dots $\\
\cline{2-3}
& $\PSL_3(2)$ & $8/21=0.380\dots$\\
\cline{2-3}
& $\Aut(\PSL_3(2))$&$23/84=0.273\dots$\\
\hline
\multirow{6}{*}{$\PSL_2(q)$} & $\PSL_2(4)\cong A_5$ & $4/15=0.266\dots$\\
\cline{2-3}
& $\PSigmaL_2(4)\cong S_5$ & $7/15=0.466$\\
\cline{2-3}
& $\PSL_2(8)$ & $8/63=0.126\dots$\\
\cline{2-3}
& $\PSL_2(16)$ & $16/255=0.062\dots$\\
\cline{2-3}
& $\PSL_2(16)\rtimes \Gal(\mathbb{F}_{16}/\mathbb{F}_{4})$ & $14/85=0.164\dots$\\
\cline{2-3}
& $\PSigmaL_2(16)$ & $169/1020=0.165\dots$\\
\hline
\multirow{6}{*}{$\PSL_n(2)$, $n\leq 5$} &  $\PSL_3(2)$&$ 8/21=0.380\dots$\\
\cline{2-3}
& $\PGL_2(7)$ & $23/84=0.273\dots$\\
\cline{2-3}
& $A_8$&$64/315=0.203\dots$\\
\cline{2-3}
& $S_8$&$389/2520=0.154\dots$\\
\cline{2-3}
& $\PSL_5(2)$&$6619/156240 = 0.042\dots$\\
\cline{2-3}
& $\Aut(\PSL_5(2))$&$10091/312480 = 0.032\dots$\\
\hline
\multirow{2}{*}{$\PSU_3(q)$, $q\geq 4$} & $\PSU_3(4)$&$64/975=0.065\dots$\\
\cline{2-3}
& $\PSigmaU_3(4)$&$1621/15600=0.103\dots$\\
\hline
\multirow{4}{*}{$\PSU_n(2)$, $n=4,5$} & $\PSU_4(2)$&$64/405 = 0.158$\\
\cline{2-3}
& $\PSigmaU_4(2)$&$ 427/3240 = 0.131$\\
\cline{2-3}
& $\PSU_5(2)$&$3019/213840 = 0.014\dots$\\
\cline{2-3}
& $\PSigmaU_5(2)$&$ 7771/427680 = 0.018\dots$\\
\hline
\multirow{6}{*}{$\PSp_4(q)$} & $\PSp_4(2)'\cong A_6$ & $17/45=0.377\dots$\\
\cline{2-3}
&  $\PSp_4(2)\cong S_6$&$ 16/45=0.355\dots$\\
\cline{2-3}
& Mathieu group $M_{10}$ & $79/180=0.438\dots$\\
\cline{2-3}
& $\PGL_2(9)$ & $43/180=0.238\dots$\\
\cline{2-3}
& $\Aut(\PSp_4(2)')$&$59/180=0.327\dots$\\
\cline{2-3} 
& $\Aut(\PSp_4(4))$&$\frac{4169}{61200}=0.068\dots<\frac{7}{15|\Out(\PSp_4(4))|}=7/60$\\
\hline
$\PSp_6(2)$ & $\PSp_6(2)$ & $1261/22680=0.055\dots$\\
\hline
\multirow{2}{*}{$\POmega_8^+(2)$} & $\POmega_8^+(2)$ & $23011/680400 = 0.0338\dots$\\
\cline{2-3}
& $\PSO_8^+(2)$&$ 32491/1360800 = 0.0238\dots$\\
\hline
\multirow{2}{*}{$F_4(q)$} & $F_4(2)$ & $\frac{18865921}{50523782400}=0.00037\dots $ \\
\cline{2-3}
&$\Aut(F_4(2))$ & $ \frac{9445213}{25261891200}=0.00037\dots$\\
\hline
\multirow{2}{*}{$G_2(2)'\cong \PSU_3(3)$} &$\PSU_3(3)$ & $ 71/756 = 0.0939\dots $\\
\cline{2-3}
&$\PSigmaU_3(3)$ & $67/756=0.0886\dots$\\
\hline
\multirow{2}{*}{$\vphantom{a}^2F_4(2)$} & $\vphantom{a}^2F_4(2)'$ &$4127/140400 = 0.02939\dots$\\
\cline{2-3}
& $\vphantom{a}^2F_4(2)$ & $4111/140400 = 0.02928\dots$\\
\hline
\end{tabular}}
\end{center}
\end{table}

\begin{thm}\label{4_crosschar}
Let $G$ be an almost simple group such that $T=\soc(G)$ is a simple group of Lie type in odd characteristic. Then $\rho_4(G)\leq 7/15$. 
\end{thm} 

\begin{proof}
Suppose that $\rho_4(G)>7/15$ and $\rho_4(G)>\rho_4(K)$ for all non-soluble $K$ with $|K|<|G|$. As in the previous theorem, we will use the notation $InnDiag(T)$, $\Phi_T$ and $\Gamma_T$ to mean the group of inner-diagonal automorphisms and the (standard) groups of field and graph automorphisms. By \autoref{Sylow_quotient} we may assume that $G/T$ is a Sylow $2$-subgroup of $\Out(T)$.

Let $P$ be a Sylow $2$-subgroup of $G$ and $Q=P\cap T$, so that $G=TP$. Then by \cite[Theorem 2]{AsOdd} (see also \cite[Theorem 4.10.6]{GLS}), if $T$ is not one of $\PSL_2(q),\ ^3D_4(q), E_6(q),\ ^2E_6(q), F_4(q), G_2(q)$ or $^2G_2(q)$, then there exist certain subgroups $R_0\unlhd R < G$ such that $P\leq R$ and $R/R_0$ satisfies the following: 
\begin{enumerate}[label={\upshape(\alph*)}]
\item If $T\cong \PSL_{2n}(q), \PSL_{2n+1}(q), \PSU_{2n}(q), \PSU_{2n+1}(q), $or $\PSp_{2n}(q)$, then $R/R_0\cong S_n$.
\item If $T\cong \POmega_{4n+1}(q),\POmega^{\pm}_{4n+2}(q), \POmega_{4n+3}(q)$, or $\POmega_{4n+4}^-(q)$, then $R/R_0$ has a normal subgroup isomorphic to $C_2^nS_n$.
\item If $T\cong \POmega_{4n}^+(q)$, then $R/R_0$ has a normal subgroup isomorphic to $C_2^{n-1}S_n$.
\item If $T\cong E_7(q)$, then $R/R_0$ has a normal subgroup isomorphic to $\PSL_3(2)$.
\item If $T\cong E_8(q)$, then $R/R_0$ has a normal subgroup isomorphic to $C_2^3\PSL_3(2)$.
\end{enumerate}
Since $R/R_0$ has a normal subgroup which has no composition factor isomorphic to $T$, $R/R_0\not \cong G$ and in particular $|R/R_0|<|G|$. By \autoref{self-normalising} and \autoref{quo}, we have $\rho_4(G)\leq \rho_4(R)\leq \rho_4(R/R_0)$. Since $G$ is assumed to have strictly larger $\rho_4$ than any non-soluble group of smaller order, $R/R_0$ must be soluble. Therefore, in cases (a)-(c), we must have $n\leq 4$, and cases (d) and (e) cannot happen. 

In addition to this, we again use \cite[Lemmas 5.4 and 6.2]{Bur}, which describe the possible soluble maximal subgroups of almost simple classical groups. By \autoref{self-normalising} and the minimality of $G$, every maximal subgroup of $G$ that contains $P$ must be soluble. The remaining possibilities of $T$ are the following.
\begin{itemize}
\item $\vphantom{a}^2G_2(q)$. Since $\Out(\vphantom{a}^2G_2(q))$ has odd order, by \autoref{Sylow_quotient} we get $\rho_4(G)\leq \rho_4(T)=\rho_4(\vphantom{a}^2G_2(q))$. A Sylow $2$-subgroup $P_0$ of $\vphantom{a}^2G_2$ has $|N_T(P_0):P_0|=21$ by \cite[Theorem 6]{KM}, so by \autoref{self-normalising}, $\rho_4(G)\leq 1/21$.
\end{itemize}
For the following groups, we will present, or show the existence of, a non-soluble subgroup satisfying the condition of \autoref{self-normalising}, thereby breaking the minimality of $G$. 
\begin{itemize}
\item $F_4(q)$. In this case $T$ has no diagonal or graph automorphisms. The field automorphisms $\Phi$ stabilize the $T$-conjugacy class of subgroup of the form $H=2\cdot \Omega_9(q)$; see \cite[Table 7]{Cr}. Since $|F_4(q)|_2 = |2\cdot \Omega_9(q)|_2$, by \autoref{normalise}, the proper non-soluble subgroup $\mathbf{N}_G(H)$ contains a Sylow $2$-subgroup of $G$. (This argument using \autoref{normalise} will be used in most of the other cases below in the same manner, so we will not mention this again for brevity.)
\item $G_2(q)$. By \cite[Theorems A and B]{Kl-G2}, there exists a maximal subgroup $\mathbf{N}_G(H)$ where $H<T$ is the centraliser of an involution in $T$, isomorphic to $(\SL_2(q)\circ \SL_2(q))\cdot C_2$, where $\circ$ denotes a central product. Since $|G_2(q)|_2 = |(\SL_2(q)\circ \SL_2(q))\cdot C_2|_2$, $H$ contains a Sylow $2$-subgroup of $G_2(q)$. Also, $G_2(q)$ has only one class of involutions, so $\Aut(T)$ stabilizes the $T$-conjugacy class of $H$. By the argument mentioned above, $\mathbf{N}_G(H)$ must be soluble, which forces $q=3$.  
\item $E_6^\pm(q)$. Here we have no diagonal automorphism of even order. By \cite[Tables 9, 10]{Cr}, there exists a non-soluble subgroup $H<T$ isomorphic to $4.(\POmega_8^+(q)\times ((q\mp 1)/2)^2/\gcd(3,q-1)).4.S_3$, whose $T$-conjugacy class is stabilized by $\Phi_T\Gamma_T$. Note that $|H|_2 = |E_6^\pm(q)|_2$, so $H$ contains a Sylow $2$-subgroup of $T$. 
\item $\vphantom{a}^3D_4(q)$. In this case there is no diagonal and graph automorphisms. By \cite[Table 8.51]{BHR}, There exists a subgroup $H<T$ isomorphic to $G_2(q)$ whose $T$-conjugacy class is stabilized by $\Phi_T$. Since $|\vphantom{a}^3D_4(q)|_2 = |G_2(q)|_2$, this contains a Sylow $2$-subgroup. $H$ is non-soluble, so this case is impossible.
\item $\PSL_n(q)$ with $3\leq n\leq 9$. First suppose that $n\neq 4, 8$. Let $a_1,a_2,a_3,a_4\in \{0,1\}$ be defined by $n=8a_1+4a_2+2a_3+a_4$. Consider the block diagonal subgroup $$B=\left(\SL_n(q)\cap \left(\GL_8(q)^{a_1}\times\GL_4(q)^{a_2}\times \GL_2(q)^{a_3}\times \GL_1(q)^{a_4}\right)\right)/\mathbf{Z}(\SL_n(q)) < \PSL_n(q).$$ (One can also use a larger reducible subgroup containing this, i.e. an appropriate subgroup in Aschbacher's class $\mathcal{C}_1$.) Note that $B$ is normalised by the conjugation by diagonal matrices, the entry-wise field automorphisms, and the inverse-transpose graph automorphism, and that $B$ contains a Sylow $2$-subgroup of $T$. By \autoref{normalise}, $\mathbf{N}_G(B)$ is a proper subgroup of $G$ which contains a Sylow $2$-subgroup of $G$. Therefore $B$ must be soluble, which happens only when $n=q=3$. 

If $n=4$ or $8$, consider the subgroup (of Aschbacher's class $\mathcal{C}_2$, ``imprimitive subgroups'')\begin{align*}
W &= \left(\SL_n(q)\cap (\GL_{n/2}(q)\wr C_2)\right)/\mathbf{Z}(\SL_n(q)) \\&= \left(\SL_n(q)\cap \langle \GL_{n/2}(q)\times \GL_{n/2}(q), \tau\rangle\right)/\mathbf{Z}(\SL_n(q))<\PSL_n(q)\end{align*}
where $\tau = \begin{pmatrix}
0 & I_{n/2} \\ I_{n/2} & 0
\end{pmatrix}.$ Again, $W$ is normalised by the conjugation by diagonal matrices, the entry-wise field automorphisms, and the inverse-transpose graph automorphism, and contains a Sylow $2$-subgroup of $T$. Therefore $W$ must be soluble, which happens only when $n=4$ and $q=3$.
\item $\PSU_n(q)$ with $3\leq n\leq 9$. As in the case of $\PSL_n(q)$, we can use the subgroup $$B=\left(\SU_n(q)\cap \left(\GU_8(q)^{a_1}\times\GU_4(q)^{a_2}\times \GU_2(q)^{a_3}\times \GU_1(q)^{a_4}\right)\right)/\mathbf{Z}(\SU_n(q)) < \PSU_n(q)$$ if $n\neq 4, 8$, and \begin{align*}
W = \left(\SU_n(q)\cap (\GU_{n/2}(q)\wr C_2)\right)/\mathbf{Z}(\SU_n(q)) <\PSU_n(q)\end{align*} if $n=4,8$. Each of these groups contains a Sylow $2$-subgroup of $T$, and their $T$-conjugacy classes are stabilized by the automorphisms of $T$, cf. \cite[Table 3.5.B]{KL}. These are soluble only when $(n,q)= (3,3)$ or $(4,3)$.
\item $\PSp_{2n}(q)$ with $2\leq n\leq 4$. Once again we use the same construction: $$B=\left(\Sp_6(q)\cap \left(\Sp_4(q)\times\Sp_2(q)\right)\right)/\mathbf{Z}(\Sp_6(q)) < \PSp_6(q)$$ if $n=3$, and \begin{align*}
W = \left(\Sp_{2n}(q)\cap (\Sp_{n}(q)\wr C_2)\right)/\mathbf{Z}(\Sp_{2n}(q)) <\PSp_{2n}(q)\end{align*} if $n=2,4$. Each of them contain a Sylow $2$-subgroup, and the conjugacy class is stabilized by the automorphisms, cf. \cite[Table 3.5.C]{KL}. The only soluble case is $(n,q)=(2,3)$. 
\item $\POmega_{2n+1}(q)$ with $3\leq n\leq 9$. By the aforementioned results in \cite{Bur}, the only case with a soluble maximal subgroup in this case is $\POmega_7(3)$. 
\item $\POmega_{2n}^+(q)$ with $4\leq n\leq 9$. By \cite{Bur}, the only cases with a soluble maximal subgroup are $n=4$ and $(n,q)=(6,3)$, $(8,3)$. When $n=4$, the subgroup $\Omega_4^+(q)^2.[4].C_2/\mathbf{Z}(\Omega_8^+(q)) < \POmega_8^+(q)$ contains a Sylow $2$-subgroup, and its $T$-conjugacy class is stabilized by all standard automorphisms, cf. \cite[Table 8.50]{BHR}. The only soluble case is $(n,q)=(4,3)$. Note that $G$ does not involve the triality automorphism.
\item $\POmega_{2n}^-(q)$ with $4\leq n\leq 10$. By \cite{Bur}, there is no soluble maximal subgroup of $G$.
\end{itemize}
We have one more type of groups, where the above arguments don't work well. We directly count the number of elements of order dividing $4$ in these groups:
\begin{itemize}
\item $\PSL_2(q)$. Here we have no graph automorphism, so $\Aut(T)=\PGammaL_2(q)$. We may assume that $\Phi_T$ consists of the entry-wise Galois automorphisms. Let $f\in \Phi_T$ be the field automorphism of the largest order among those appearing in some element of $G$, so that $G\leq \PGL_2(q)\rtimes \langle f\rangle$. We may assume that $f$ maps each entry of a matrix to its $p^e$th power for some $e\in \mathbb{Z}_{\geq 0}$. Since $f$ must have order dividing $4$, we have $q\in\{p^e, p^{2e}, p^{4e}\}$.

First suppose that $q\equiv 1$ mod $4$. Consider the subgroups \begin{align*}
& C_f = \left\langle \begin{pmatrix}
\beta & 0 \\ 0 & \beta^{-1}
\end{pmatrix}\right\rangle/\mathbf{Z}(\SL_2(q)) \leq  C = \left\langle \begin{pmatrix}
\alpha & 0 \\ 0 & \alpha^{-1}
\end{pmatrix}\right\rangle/ \mathbf{Z}(\SL_2(q))<\PSL_2(q) ,\\&M = \left\langle \begin{pmatrix}
\alpha & 0 \\ 0 & \alpha^{-1}
\end{pmatrix},  \begin{pmatrix}
0 & 1 \\ -1 & 0
\end{pmatrix}\right\rangle / \mathbf{Z}(\SL_2(q)) < \PSL_2(q)\\& D =\left\langle \begin{pmatrix}
\alpha & 0 \\ 0 & \alpha^{-1}
\end{pmatrix},  \begin{pmatrix}
0 & 1 \\ -1 & 0
\end{pmatrix}, \begin{pmatrix}
\alpha & 0 \\ 0 & 1
\end{pmatrix} \right\rangle / \mathbf{Z}(\GL_2(q)) < \PGL_2(q) 
\end{align*} where $\alpha$ is a primitive element of $\mathbb{F}_q$ and $\beta$ is a primitive element of the fixed field of $f$. $C\cong C_{(q-1)/2}$, $M\cong D_{q-1}$ is dihedral of order $q-1$, and $D\cong D_{2(q-1)}$ is dihedral of order $2(q-1)$. Note that $|\PSL_n(q)|_2 = |(q-1)(q+1)/\gcd(2,q-1)|_2 = |q-1|_2$, so $M$ contains a Sylow $2$-subgroup of $T$. $\Phi_T$ normalises $M$ and $D$, so $D\Phi_T$ is a subgroup of $\PGammaL_2(q)$ that contains a Sylow $2$-subgroup of $\PGammaL_2(q)$. We may assume that $P\leq D\langle f\rangle$. If $\PGL_2(q)\leq G$ then $D\langle f\rangle \leq G$, so it satisfies the condition of \autoref{self-normalising}. If $\PGL_2(q)\not\leq G$, then $G\cap \PGL_2(q) = \PSL_2(q)$, so $P\cap \PGL_2(q) \leq \PSL_2(q)\cap D= M$, hence $M\langle df\rangle \leq G$ for some diagonal automorphism $d$ which always has order $2$. In this case $|G|=|\PSL_2(q)|o(f)$ and $|M\langle df\rangle|=|M|o(f)$, so we can apply \autoref{self-normalising} to $M\langle df\rangle$.

An element of $D$ can be written uniquely as the image in $\PGL_n(q)$ of a matrix in one of the following forms: $$\diag(1,x)=\begin{pmatrix}
1 & 0 \\ 0 & x
\end{pmatrix},\ \adiag(1,x)=\begin{pmatrix}
0 & 1 \\ x & 0
\end{pmatrix}$$ for some $x\in \mathbb{F}_q^\times$. Note that for each $x\in \mathbb{F}_q^\times$ and $1\leq r< o(f)$,
\begin{align*}
(\diag(1,x)f^r)^4 = \diag(1,x^{1+p^{er}+p^{2er}+p^{3er}}) 
\end{align*}
so $\diag(1,x)f$ has order dividing $4$ if and only if $x^{1+p^{er}+p^{2er}+p^{3er}}=1$. If $q=p^{e\gcd(4,r)}$, then this just means $x^4=1$. If $q=p^{2e\gcd(4,r)}$, then we get $x^{2(1+p^{er})} = 1$, so there are $2(p^{er}+1)$ such $x$. If $q=p^{4e\gcd(4,r)}$, then $r\in\{1,3\}$, $q=p^{4e}$ and $x^{1+p^e+p^{2e}+p^{3e}}=1$, so there are $1+p^e+p^{2e}+p^{3e}$ such $x$. 

For the second type of elements, we have 
\begin{align*}
(\adiag(1,x)f^r)^4 = \diag(1, x^{(1+p^{2er})(p^{er}-1)})\text{ (as elements of }\PGammaL_2(q))
\end{align*}
so $\adiag(1,x)f^r$ has order dividing $4$ if and only if $x^{(1+p^{2er})(p^{er}-1)}=1$. If $q=p^{e\gcd(4,r)}$ then this just becomes $1=1$, so it is true for all $x$. If $q=p^{2e\gcd(4,r)}$ then this means $x^{2(p^{er}-1)}=1$. If $q=p^{4e\gcd(4,r)}$ then $r=1$ or $3$, $q=p^{4e}$ and $x^{(1+p^{2e})(p^{e}-1)}=1$, so there are $(p^{2e}+1)(p^e-1) = p^{3e}-p^{2e}+p^e-1$ such $x$. 

If $q=p^{4e}$, then by \autoref{self-normalising}, 
\begin{align*}
&\rho_4(D\langle f\rangle)\\=& \frac{4 + 2(p^{2e}+1) + 2(1+p^e+p^{2e}+p^{3e}) + (q-1) + 2(p^{2e}-1) + 2(p^{3e}-p^{2e}+p^e-1)}{8(q-1)} \\=& \frac{q +4p^{3e} + 4p^{2e} + 4p^e + 3}{8(q-1)}
\end{align*} 
which is less than $7/15$ for all $p^e$. In particular, if $\PGL_2(q)\leq G$ then $\rho_4(G)\leq \rho_4(D\langle f\rangle)<7/15$. If $\PGL_2(q)\not\leq G$, then only half of the $q-1$ elements of the form $\adiag(1,x)$ lies in $M$, so \begin{align*}
&\rho_4(G)\leq \rho_4(M\langle df\rangle) \\\leq& \frac{4 + 2(p^{2e}+1) + 2(1+p^e+p^{2e}+p^{3e}) + (q-1)/2 + 2(p^{2e}-1) + 2(p^{3e}-p^{2e}+p^e-1)}{4(q-1)} \\=& \frac{q +8p^{3e} + 8p^{2e} + 8p^e + 7}{8(q-1)}
\end{align*} 
which is less than $7/15$ when $p^e\geq 5$.

If $q=p^{2e}$, then
\begin{align*}
&\rho_4(D\langle f\rangle) = \frac{4 + 2(p^{e}+1) + (q-1) + 2(p^{e}-1)}{4(q-1)} = \frac{q + 4p^{e} + 3}{4(q-1)}
\end{align*}
which is $\leq 7/15$ when $p^e\geq 7$. Therefore, if $\PGL_2(q)\leq G$ then $\rho_4(G)\leq \rho_4(D\langle f\rangle)\leq 7/15$. If $\PGL_2(q)\not< G$, then again $M\langle df\rangle$ contains only half of the elements of the form $\adiag(1,x)$, so $$\rho_4(G)\leq \rho_4(M\langle df\rangle) \leq \frac{4 + 2(p^e+1) + (q-1)/2 + 2(p^e-1)}{2(q-1)} = \frac{q + 8p^e + 7}{4(q-1)}$$ which is $\leq 7/15$ when $p^e\geq 11$.

If $q=p^e$ (so $f=1$), then $G$ does not involve any nontrivial field automorphism, so it is either $\PSL_2(q)$ or $\PGL_2(q)$. In this case it is easy to count the number of elements of order dividing $4$ using \cite[Lemma 2.4]{Green}; it was done in \autoref{PSL} for $\PSL_2(q)$, and for $\PGL$ we have $$\rho_4(\PGL_2(q)) = 
\frac{1+(3-c_1)(q+1)q/2+(3/2-c_2)q(q-1)}{(q-1)q(q+1)} \leq \frac{1+3q^2}{(q-1)q(q+1)} $$ where $c_1=0$ if $4\mid q-1$ and $c_1=2$ otherwise, and $c_2 = |\{x\in \mathbb{F}_q\mid x^4=-4\}|/4$. The right-hand side is $\leq 7/15$ when $q\geq 7$, and for $q=5$ we have exactly $\rho_4(\PGL_2(5))=7/15$. $\PSL_2(3)$ is soluble, so we exclude $q=3$.

When $q\equiv 3$ mod $4$, $q$ must be an odd power of $p$, so there is no field automorphism of order $2$ or $4$. Therefore we again get $G\leq \PGL_2(q)$, so $\rho_4(G)\leq 7/15$ except for the cases excluded above. The remaining cases are $q=3^2, 5^2, 7^2, 3^4$, with the additional condition $\PGL_2(q)\not\leq G$ for $q=7^2$ and $3^4$.
\end{itemize}

We computed $\rho_4(G)$ for each of the groups $G$ we found above and listed them in \autoref{4crosschartable}. Again, for some $T$, we only computed $\rho_4$ for some large $G$ when $\rho_4(G)\leq 7/(15|G:T|)$, which forces $\rho_4(H)\leq 7/15$ for all $T\leq H\leq G$. 
\end{proof}

\begin{table}[h]
\begin{center}
\caption{$\rho_4(G)$ for $G$ of Lie type in odd characteristic in the proof of \autoref{4_equichar}.}\label{4crosschartable}
\begin{tabular}{|c|c|c|}
\hline
$T$ & $G$ & $\rho_4(G)$\\
\hline
\multirow{2}{*}{$G_2(q)$} & $G_2(3)$ & $5989/265356=0.022\dots$\\
\cline{2-3}
& $\Aut(G_2(3))$ & $17221/530712=0.032\dots$\\ 
\hline
\multirow{3}{*}{$\PSL_n(q)$, $3\leq n\leq 9$} & $\Aut(\PSL_4(3))$ &$\frac{30611}{758160}=0.040\dots < \frac{7}{15|\Out(\PSL_4(3))|} = 7/60$ \\
\cline{2-3}
& $\Aut(\PSL_3(3))$ & $161/1404=0.114\dots$\\
\cline{2-3}
& $\PSL_3(3)$ & $205/1404=0.146\dots$ \\
\hline
\multirow{12}{*}{$\PSL_2(q)$} & $\PGammaL_2(81)$ & $\frac{1475}{26568}=0.055\dots< \frac{7}{15|\Out(\PSL_2(81))|}=7/60 $ \\
\cline{2-3}
& $\PGammaL_2(49)$ & $\frac{1397}{14700}=0.095\dots < \frac{7}{15|\Out(\PSL_2(49))|} =7/60$  \\
\cline{2-3}
& $\PGammaL_2(25)$ & $97/650=0.149\dots$\\
\cline{2-3}
& Intermediate groups & $319/3900=0.081\dots$ \\
\cline{3-3}
& between $\PSL_2(25)$ & $191/975=0.195\dots$\\
\cline{3-3}
& and $\PGammaL_2(25)$ & $569/3900=0.145\dots$\\
\cline{2-3}
& $\PSL_2(25)$ & $122/975=0.125\dots$\\
\cline{2-3}
& $\PGammaL_2(9)\cong \Aut(S_6)$ & $59/180=0.327\dots$\\
\cline{2-3}
& $S_6$ & $ 16/45=0.355\dots$\\
\cline{2-3}
& $\PGL_2(9)$ & $43/180=0.238\dots$ \\
\cline{2-3}
& Mathieu group $M_{10}$ & $79/180=0.438\dots$ \\
\cline{2-3}
& $\PSL_2(9)\cong A_6$ & $17/45=0.377\dots$\\
\hline
\multirow{4}{*}{$\PSU_n(q)$} & $\PSU_3(3)$ &$71/756=0.093\dots$ \\
\cline{2-3}
& $\PSigmaU_3(3)$ & $67/756=0.088\dots$\\
\cline{2-3}
& $\PSU_4(3)$ & $15061/204120=0.073\dots$\\
\cline{2-3}
& $\PGammaU_4(3)$ & $2069/51030=0.040\dots$ \\
\hline
\multirow{2}{*}{$\PSp_{2n}(q)$} & $\PSp_4(3)$ & $64/405=0.158\dots$ \\
\cline{2-3}
& $\Aut(\PSp_4(3))$ & $ 427/3240=0.131\dots$\\
\hline
\multirow{2}{*}{$\POmega_{2n+1}(q)$} & $\POmega_7(3)$ & $\frac{652339}{71646120}=0.0091\dots$\\
\cline{2-3}
&$\text{PSO}_7(3)$ & $\frac{313657}{35823060}=0.0087\dots$ \\
\hline
$\POmega_{2n}^+(q)$ & $\text{PGO}_{8}^+(3)\rtimes C_2$ & $\frac{283169599}{154755619200}=0.0018\dots< \frac{7}{15|\Out(\POmega_8^+(3))|_2} = \frac{7}{120}$ \\
\hline
\end{tabular}
\end{center}
\end{table}

Therefore, by \autoref{almost_simple_reduction}, \autoref{4_sporad}, \autoref{4_equichar} and \autoref{4_crosschar}, we can complete the proof of \autoref{mainthm1}.

\begin{proof}[Proof of \autoref{mainthm1}]
By \autoref{Prime_power}, $\epsilon_k\leq 1/2$ for all odd primes $k$, as well as powers of all primes other than $2$ and $3$. For $k=2$, as discussed in the introduction, the conjecture was proved by Wall \cite{Wall}, and further studied by Liebeck and MacHale \cite{LM}, Mann \cite{Mann}, and Berkovich \cite{Berkovich}. For $k=4$, we got $\epsilon_4=\epsilon_4^*=\rho_4(S_5)=7/15<1/2$ by \autoref{almost_simple_reduction}, \autoref{4_sporad}, \autoref{4_equichar} and \autoref{4_crosschar}. For $k=2^e$ with $e\geq 3$ and $k=3^e$ with $e\geq 2$, \autoref{M10}, \autoref{epsilon_9} and \autoref{examples_summary} shows that the conjecture fails. This completes the proof of part (1). Part (2) follows from \autoref{PSigmaL2} and \autoref{examples_summary}. Part (3) just says that there are many other examples as we saw in Section 3, which can be found in \autoref{smallepsilontable}.
\end{proof}

\section{Open questions}

Here we list some questions that we couldn't study in this paper.

\begin{qst}
{\upshape(1)} Is $\epsilon_3 = 7/20$?\\{\upshape(2)} Is there any $k$, especially prime powers, such that $\epsilon_k$ is exactly $1/2$?\\{\upshape(3)} Does the conjecture fail for every $k>1$ which is not a prime power? Equivalently, for primes $p,q>3$, when is $\epsilon_{pq}> 1/2$? \\{\upshape(4)} Does $\epsilon_k^*$ exist for all $k$? As a special case, is there a $k$ such that $\epsilon_k=1$ but no finite non-soluble group has exponent dividing $k$? \\{\upshape(5)} Is there a choice of $k\in \mathbb{Z}_{>0}$ and $r\in (0,1)$ such that there are infinitely many finite non-soluble monolithic groups $G$ with $\rho_k(G)\geq r$? What about $\rho_k^*(G)$?
\end{qst}

\section*{Acknowledgement}
We thank Hung P. Tong-Viet for helpful discussion.

\begin{table}[h]
\begin{center}
\caption{Table of $\epsilon_k$ for some small $k$}
\label{smallepsilontable}
\footnotesize
\begin{tabular}{|c|c|c|c|c|c|}
\hline 
$k$ & $\epsilon_k$ & $\epsilon_k^*$ &$>1/2$ & Example & Remark \\ 
\hline 
$1$ & $1/60$ & $\exists$ & No & $A_5$ &  \\ 
\hline 
$2$ & $4/15$ & $\exists$ & No & $A_5$ & \cite{Berkovich} \\ 
\hline
$3$ & $\in \left[7/20, 1/2\right]$  & & No &$A_5$ & \\
\hline
$4$ & $7/15 $ & $\exists$ & No & $S_5$ & \\
\hline
$5$ & $\in \left[5/12, 1/2\right]$ & & No & $A_5$ & \\
\hline
$6$ & $\geq 3/5$ & & Yes & $A_5$ &  \\ 
\hline 
$7$ & $\in \left[\frac{67}{156}\approx 0.43, 1/2\right]$ & & No & $\PSL_2(13)$ & \\
\hline
$8$ & $\geq 31/45$ & & Yes & Mathieu group $\text{M}_{10}$ & \\
\hline
$9$ & $\geq 191/364\approx 0.52$ & & Yes & $\PSigmaL_2(3^3)$ & \\
\hline
$10$ & $\geq 2/3$ & & Yes & $A_5$ &  \\ 
\hline 
$11$ & $\in \left[\frac{251}{552}\approx 0.45, 1/2\right]$ & & No & $\PSL_2(23)$ & \\
\hline
$12$ & $\geq 4/5$ & & Yes & $S_5$ & \\
\hline
$13$ & $\in \left[\frac{277}{600}\approx 0.46,1/2\right]$ & & No & $\PSL_2(25)$ & \\
\hline 
$15$ & $\geq 3/4$ & & Yes & $A_5$ & \\
\hline
$24$ & $\geq 21/25$ & & Yes & Index 2 subgroup of $S_5\wr C_2$ & not almost simple \\
\hline
$30n$, $n\in \mathbb{Z}_{>0}$ & $1$ & $\exists$ & Yes & $A_5$ & exponent \\
\hline
$35$ & $\geq 395/576\approx 0.68 $ & & Yes & $\PSL_3(4)=\text{M}_{21}$ & \\
\hline
$84n$, $n\in \mathbb{Z}_{>0}$ & $1$ & $\exists$ & Yes & $\PSL_2(7)=\PSL_3(2)$ & exponent\\
\hline
$91=7\times 13$ & $\geq 211/320\approx 0.66$ & & Yes & $\text{Sz}(8)$ & \\
\hline
$155=5\times 31$ & $\geq \frac{109971}{209920}\approx 0.52$ & & Yes & $\text{Sz}(32)$ & \\
\hline
$p^a$, $p\geq 5$ prime & $\in ((p-1)/(6p),1/2]$ & & No &  $\PGammaL_2(2^p)$ & \\
\hline
$2p$, $p> 3$ prime & $\geq \frac{2^{p}}{p(2^p-1)(2^p+1)} +  \frac{2(p-1)}{3p}$ & & Yes & $\PGammaL_2(2^p)$ & \\
\hline
$3p$, $p> 2$ prime & $\geq \frac{2\cdot 3^p}{p(3^p-1)(3^p+1)} +\frac{3(p-1)}{4p}$ & & Yes & $\PSigmaL_2(3^p)$ & \\
\hline
\end{tabular} 
\end{center}
\end{table}

\begin{table}[h]
\begin{center}
\caption{$\rho_4$ for sporadic almost simple groups}\label{rho4table}
\begin{tabular}{|c|c|}
\hline 
Group & $\rho_4$   \\ 
\hline 
$M_{11}$ & $1156/7920 = 0.1459\dots$\\
\hline
$M_{12}$ & $6832/95040 = 0.07188\dots$\\
\hline
$\Aut(M_{12})$ & $19504/190080 = 0.1026\dots$\\
\hline
$M_{22}$ & $42736/443520 = 0.0963\dots$\\
\hline
$\Aut(M_{22})$ & $67552 / 887040 = 0.0761\dots$\\
\hline
$M_{23}$ & $322576/10200960 = 0.0316\dots$\\
\hline
$M_{24}$ & $5143744/244823040 = 0.0210\dots$\\
\hline
$J_1$ & $1464/175560 = 0.0083\dots$\\
\hline
$J_2$ & $9136/604800 = 0.0151\dots$\\
\hline
$\Aut(J_2)$ & $73936/1209600 = 0.0611\dots$\\
\hline
$J_3$ & $549424/50232960 = 0.0109\dots$\\
\hline
$\Aut(J_3)$ & $1616464/100465920 = 0.0160\dots$ \\
\hline
$J_4$ & $2916489237917696/86775571046077562880 = 0.00003\dots$ \\
\hline
$Co_1$ & $379907764350976/4157776806543360000 = 0.00009\dots$ \\
\hline
$Co_2$ & $48644632576/42305421312000 = 0.0011\dots$\\
\hline
$Co_3$ & $347061376/495766656000 = 0.0007\dots$\\
\hline
$Fi_{22}$ & $42303609856 / 64561751654400 = 0.0006\dots$\\
\hline
$\Aut(Fi_{22})$ & $102513958912 / 129123503308800 = 0.0007\dots$ \\
\hline
$Fi_{23}$ & $168350817359872 / 4089470473293004800 = 0.00004\dots $\\
\hline
$Fi_{24}'$ & $4503927670511693824/1255205709190661721292800 =  0.000003\dots$ \\
\hline
$\Aut(Fi_{24}')$ & $7260932825889366016/2510411418381323442585600 = 0.000002\dots $ \\
\hline
$HS$ & $898976/44352000= 0.0202\dots$\\
\hline
$\Aut(HS)$ & $2632576/88704000 = 0.0296\dots$\\
\hline
$McL$ & $9377776/898128000 = 0.0104\dots$ \\
\hline
$\Aut(McL)$ & $10738576/1796256000 = 0.0059\dots$ \\
\hline
$He$ & $48193216/4030387200 = 0.0119\dots$ \\
\hline
$\Aut(He)$ & $65253056/8060774400= 0.0080\dots$   \\
\hline
$Ru$ & $486366976/145926144000 = 0.0033\dots$\\
\hline
$Suz$ & $2007237376/448345497600 = 0.0044\dots$\\
\hline
$\Aut(Suz)$ & $2871741952/896690995200 = 0.0032\dots  $ \\
\hline
$ON$ & $1808632288/ 460815505920 = 0.0039\dots$  \\
\hline
$\Aut(ON)$ &  $1811257120/921631011840 = 0.0019\dots $ \\
\hline
$HN$ & $ 136354053376/ 273030912000000 = 0.000499\dots $ \\
\hline
$\Aut(HN)$ & $367513221376/546061824000000 = 0.000673\dots$ \\
\hline
$Ly$ & $2569014039376 /51765179004000000 =  0.000049\dots$ \\
\hline
$Th$ & $12051296978176 / 90745943887872000 = 0.00013\dots$ \\
\hline
$B$ & $\frac{76622052349349502046437376}{ 4154781481226426191177580544000000}= 1.8\dots\times 10^{-8}$ \\
\hline
$M$ &  $\frac{98008936868544666550542251647672320000000}{ 114629592302884744244097149551190921445376} = 1.4\dots\times 10^{-13}$\\
\hline
\end{tabular} 
\end{center}
\end{table}

\end{document}